\documentclass[11pt, reqno]{amsart}
\usepackage[margin=1.2in]{geometry}                % See geometry.pdf to learn the layout options. There are lots.
\usepackage{graphicx}
\usepackage{amssymb,amscd,amsthm}
\usepackage{epstopdf}
\usepackage{tikz}
\usepackage[utf8]{inputenc}
\usepackage[colorlinks=true, pdfstartview=FitV, linkcolor=blue, citecolor=blue, urlcolor=blue]{hyperref}
\numberwithin{equation}{section}
\usepackage{tikz-cd}
\usepackage[nameinlink]{cleveref}
\usepackage{comment}

\newcommand{\kk}{\mathbb{K}}

\newcommand{\cG}{{\mathcal G}}

\def\Ass{\operatorname{Ass}}
\DeclareMathOperator{\codim}{ht}

\DeclareMathOperator{\reg}{reg}

\theoremstyle{plain} %% This is the default, anyway

\newtheorem{thm}{Theorem}[section]

\newtheorem*{introthm*}{Theorem}

\newtheorem{cor}[thm]{Corollary}
\newtheorem{lem}[thm]{Lemma}
\newtheorem{prop}[thm]{Proposition}

\theoremstyle{definition}
\newtheorem{defn}[thm]{Definition}

\newtheorem{ex}[thm]{Example}

\newtheorem{rem}[thm]{Remark}

\theoremstyle{remark}

\title{Three invariants of geometrically vertex decomposable ideals}

\author{Th\'ai Th\`anh Nguy$\tilde{\text{\^E}}$n}
\address{Department of Mathematics and Statistics, McMaster University, Canada \\
and University of Education, Hue University, 34 Le Loi St., Hue, Viet Nam}
\email{nguyt161@mcmaster.ca}
\author{Jenna Rajchgot}
\address{Department of Mathematics and Statistics, McMaster University, Canada}
\email{rajchgoj@mcmaster.ca}
\author{Adam Van Tuyl}
\address{Department of Mathematics and Statistics, McMaster University, Canada}
\email{vantuyl@math.mcmaster.ca}

\begin{document}

\begin{abstract}
    We study three invariants of geometrically vertex decomposable ideals:  
    the Castelnuovo-Mumford regularity, the multiplicity, and 
    the $a$-invariant. % of geometrically vertex decomposable ideals.  
    We show that these invariants
    can be computed recursively using the
    ideals that appear in the geometric
 vertex decomposition process. 
    
    As an application, we prove that the $a$-invariant of a geometrically vertex decomposable ideal is non-positive. We also recover 
    some previously known results in the literature including a formula for the regularity of the Stanley--Reisner ideal of a
    pure vertex decomposable simplicial complex, and proofs that some well-known families of ideals are Hilbertian. Finally, we apply our recursions to the study of toric ideals of bipartite graphs. Included among our results on this topic is a new proof for a known bound on the $a$-invariant of a toric ideal of a bipartite graph. %, can be recovered from our formulae.
\end{abstract}

\maketitle
%%%%%%%%%%%%%%%%%%%%%%%%%%%%%%%%%%

\section{Introduction}
\label{sec.intro}
Vertex decomposable simplicial complexes and their %algebraic counterparts, 
associated Stanley-Reisner ideals have been extensively studied in the fields of combinatorial algebraic topology and combinatorial commutative algebra (for example, see \cite{DE09,HW2014,KMY09,PB80,W09}).
These complexes are known to have many nice combinatorial properties. Such complexes are defined recursively via vertex decompositions into subcomplexes, and this suggests that one can study their structure and invariants by means of those recursions. A generalization of this concept, \emph{geometric vertex decomposition}, was introduced by A. Knutson, E. Miller and A. Yong in \cite{KMY09}. They used this Gr\"obner degeneration technique to study Schubert determinantal ideals associated to vexillary permutations. 

Building upon this work, P. Klein and the second author in \cite{KR21} introduced the notion of a \emph{geometrically vertex decomposable ideal} which is a generalization of the Stanley-Reisner ideal of a vertex decomposable simplicial complex. In fact, a geometrically vertex decomposable squarefree monomial ideal is precisely the Stanley-Reisner ideal of a vertex decomposable simplicial complex with the geometric vertex decomposition given by the vertex decomposition of the complex.  Other well known families of 
geometrically vertex decomposable ideals include  Schubert determinantal ideals \cite{KR21} and toric
ideals of bipartite graphs \cite{CSRV22}.
The technique of geometric vertex decomposition has been increasingly useful in various algebro-geometric contexts including applications in liaison theory \cite{KR21}, Gr\"{o}bner geometry of Schubert varieties  \cite{KMY09, MR4635090, KW21},
and the study of Hessenberg varieties \cite{CDHR23,DH}. It was 
also shown in \cite{CSRV22,KR21} that geometrically vertex decomposable ideals have many algebraic properties in common with Stanley-Reisner ideals of vertex decomposable simplicial complexes. 

The purpose of this paper is to study algebraic invariants of geometrically vertex decomposable ideals. We shall exploit their inherently recursive structure to derive recursive formulae for their invariants. Our recursions reduce the study of our original geometrically vertex decomposable ideal to the study of two related geometrically vertex decomposable ideals, each of which is in one less variable. % that are still geometrically vertex decomposable. 
 The following theorem summarizes our main results about these invariants. In the statement below, the ideals $C_{y,I}$ and 
$N_{y,I}$ refer to ideals formed from the decomposition of $I$;
complete definitions are found in Section 2. Furthermore, $\reg(R/I)$, $e(R/I)$, and $a(R/I)$ refer to the Castelnuovo-Mumford regularity, multiplicity, and $a$-invariant of $R/I$ respectively.

\begin{thm}\label{mainthm}
Suppose that $I \subseteq R = \mathbb{K}[x_1,\ldots,x_n]$ is a homogeneous, geometrically 
vertex decomposable ideal and $\text{in}_y(I) = C_{y,I} 
\cap (N_{y,I}+\langle y \rangle)$ is its geometric vertex 
decomposition. If the decomposition is non-degenerate, then
\begin{enumerate}
    \item {\rm (Theorem  \ref{thm.regformula})}
$
\reg(R/I)=\max \{\reg(R/N_{y,I}), 
\reg(R/C_{y,I})+1 \},$
\item {\rm (Theorem \ref{thm.eformula})}
$e(R/I) = e(R/N_{y,I})+e(R/C_{y,I}),$ and
\item {\rm (Theorem \ref{thm.aformula})}
$a(R/I) = \max\{a(R/N_{y,I})+1,a(R/C_{y,I})+1\}$.
\end{enumerate}
When the geometric vertex decomposition is degenerate, we have 
\[
\reg(R/I)=\reg(R/N_{y,I}), ~~e(R/I)=e(R/N_{y,I}), \text{  and  } a(R/I)=a(R/N_{y,I}).
\]
\end{thm}
\noindent
A key observation of A. Knutson, 
E. Miller, and A. Yong \cite{KMY09} is that the Hilbert series of a geometric vertex decomposition of $I$  is related to the Hilbert series of the smaller ideals (see Theorem \ref{thm.hilbertseries}).  
We use this result to show that in the case
that $I$ is also geometrically vertex decomposable,
there is relation among the associated
$h$-polynomials (see Theorem \ref{thm.hilbertseries2}).  
The proof of Theorem \ref{mainthm} then relies heavily on this
relation among the $h$-polynomials.
%The Hilbert series then allows us to deduce results on 
%the Castelnuovo-Mumford regularity, $a$-invariant, and multiplicity
%for geometrically vertex decomosable ideals.

Applying these recursions to various classes of ideals that are previously known to be geometrically vertex decomposable (see \cite{CSRV22,KR21}), we are able to offer a new approach to recover several known results.  As an example,
the recursive formula for the regularity of
Stanley-Reisner ideals of pure vertex decomposable simplicial complexes,
independently found by H.T. H\`a and R. Woodroofe \cite{HW2014} and
S. Moradi and F. Khosh-Ahang \cite{MK16}, can be deduced from
Theorem \ref{mainthm} (see Corollary \ref{cor.simplicialcomplexcase}).

Theorem \ref{mainthm} can be also be used to show that all geometrically vertex decomposable ideals are ``almost'' Hilbertian.  For a homogeneous ideal $I \subseteq R$,
let $HP_{R/I}(t)$ denote the Hilbert polynomial of $R/I$ and
let $HF_{R/I}(t)$ denote its Hilbert function.  An ideal
is {\it Hilbertian} if $HP_{R/I}(t) = HF_{R/I}(t)$ for all $t \geq 0$; this definition is attributed
to S. Abhyankar (see \cite{AK1989}).
Recently, A. Stelzer and A. Yong \cite{SY23} proved
that Schubert determinatal ideals are Hilbertian.  Because the $a$-invariant
is intimately linked to when $HF_{R/I}(t)$ and $HP_{R/I}(t)$ agree, we can contribute the following result:

\begin{thm}[Corollary \ref{cor.nonpos}]\label{mainthm2}
Let $I\subset R$ be a proper, homogeneous, geometrically vertex decomposable ideal. Then  $a(R/I)\leq 0$.  Consequently,
$HF_{R/I}(t) = HP_{R/I}(t)$ for all $t \geq 1$.
\end{thm}

\noindent
Note that Theorem \ref{mainthm2} implies that, except possibly at 
$t=0$, the Hilbert function and the Hilbert polynomial of
a geometrically vertex decomposable ideal agree. Under extra
hypotheses, we are able to determine when there is also agreement
at $t=0$, and consequently, the ideal $I$ is Hilbertian.

In the last part of the paper, we  apply our results to the class of toric ideals 
of bipartite graphs, which are known to be
geometrically vertex decomposable by \cite{CSRV22} (which builds on \cite{CG}). 
Our results complement and extend recent work on invariants of toric ideals of graphs; for example, see \cite{ADS22,BOVT17,CN09,DA2015,GHKKPVT,HBO2019,OH,TT2011,Vart}.
Recall that if $G = (V,E)$
is a finite simple graph with vertex set $V = \{x_1,\ldots,x_n\}$ and 
edge set $E = \{e_1,\ldots,e_q\}$, the {\it toric ideal of $G$}, denoted $I_G$ is 
the kernel of the $\mathbb{K}$-algebra homomorphism $\varphi:\mathbb{K}[e_1,\ldots,e_q] \rightarrow \mathbb{K}[x_1,\ldots,x_n]$ given by
$\varphi(e_i) = x_jx_k$ where $e_i = \{x_i,x_j\} \in E$.
 Theorems \ref{mainthm} and \ref{mainthm2}  allow us to show the
 following results:

\begin{thm}[Theorems \ref{thm.regbipartite} and \ref{thm.toricbipartitehilbertan}]
\label{thm.regbipartiteintro}
Let $H$ be any subgraph of a bipartite graph $G$.  Then
\begin{enumerate}
        \item $\reg(I_H) \le \reg(I_G)$,
        \item $a(\kk[E(G)]/I_H) \leq a(\kk[E(G)]/I_G)$, and 
        \item $e(\kk[E(G)]/I_H) \leq 
        e(\kk[E(G)]/I_G)$.
    \end{enumerate}
Furthermore, if $G$ is connected, then $I_G$ is Hilbertian.
\end{thm}

Theorem \ref{thm.regbipartiteintro} (1) was recently shown in \cite[Theorem 6.11]{ADS22} by A. Almousa, A. Dochtermann, and B. Smith using combinatorial techniques involving root polytopes and also in \cite[Corollary 8.16]{VV23} by M. V. Pinto and R. H. Villarreal using edge polytopes. Note that one could use the above results to obtain the upper bounds in terms of graph-theoretic invariants as those invariants of the complete bipartite graphs can be exactly computed (by the technique of geometric vertex decomposition or other techniques). Our technique not only gives a new proof for the regularity bound but can also be used to recover the results on precise values of the regularity, $a$-invariant, as well as multiplicity of toric ideals of Ferrer graphs (including complete bipartite graphs) in \cite{CN09} by A. Corso and U. Nagel.

\begin{thm}[Theorem \ref{thm.Ferrer}]
\label{thm.Ferrerintro}
Let $\lambda = (\lambda_1, \lambda_2, \ldots ,\lambda_n)$ 
be a partition with $\lambda_1\ge \lambda_2 \ge \cdots \ge 
\lambda_n$ and let $T_\lambda$ be the associated Ferrers graph. 
Denote $I_\lambda = I_{T_\lambda}$ and $R = \kk[E(T_\lambda)].$
\begin{enumerate}
    \item If $n=1$ or $\lambda_2=1$, then $\reg(R/I_\lambda)=0$.
    \item If $\lambda_2 \ge 2$, and suppose that $\lambda = 
    (\lambda_1, \ldots , \lambda_s ,1,1,\ldots ,1)$ where 
    $\lambda_s \ge 2$, then 
    $$\reg(R/I_\lambda) = \min \{ s-1, \{\lambda_j+j-3 \ | \ 2\le 
    j \le s \} \}.$$
\end{enumerate}
In particular, if $G=K_{n,m}$ is a complete bipartite graph, then $\reg(R/I_G) = \min \{ n,m \}-1.$
\end{thm}

\noindent
In addition to the above results, we consider the ``gluing'' procedure of 
G. Favacchio, J. Hofscheier, G. Keiper and the last author \cite{FHKV21}
which ``glues'' an even cycle to a graph $G$ to form a new graph $H$. We
relate the regularity, the $a$-invariant, and the multiplicity of 
$I_G$, when this ideal is geometrically vertex decomposable ideal, to that
of $I_H$ (see Theorem \ref{thm.gluecycle} and Corollary \ref{cor.aegluecycle}).  To further illustrate the
usefulness of our results,
our techniques are used to explicitly compute all three invariants for all toric ideals
of graphs that belong to a family first considered in \cite{GHKKPVT}.

Our paper is structured as follows.  In Section 2 we recall the relevant
background on geometrically vertex decomposable ideals.  In Sections 3 through
5, we consider the regularity, the multiplicity, and the $a$-invariant, 
respectively.  In the last section, we apply our results
to study the invariants of toric ideals of (bipartite) graphs.

%%%%%%%%%%%%%%%%%%%%%%%%%%%%%%%%%%%

\section{Background on Geometrically Vertex Decomposable Ideals}
\label{sec.gvd}

In this section we  recall the notion of geometric 
vertex decomposition, introduced by A. Knutson, E. Miller 
and A. Yong in \cite{KMY09}, and geometrically vertex 
decomposable ideals, introduced by P. Klein and the second 
author in \cite{KR21}. Due to the recursive nature of the 
theory of geometric vertex decomposition, this technique 
provides us with a convenient inductive set-up for studying 
properties and invariants of certain classes of ideals, as expanded upon in later sections.

%More specifically, for geometrically vertex decomposable 
%ideals, we have a very nice recursive formula for the 
%regularity that allows us to compute the regularity of 
%various classes of ideals (e.g., see the following 
%sections).  

Throughout, $\kk$ will be an algebraically closed field of 
characteristic $0$ and $R=\kk[x_1,\ldots ,x_n]$ will be a 
polynomial ring in $n$ variables. Fix a variable $y=x_j$. 
Then for any $f\in R$, we can write $f=\sum_{i=0}^d 
\alpha_i y^i$, where for each $i$, $\alpha_i \in 
\kk[x_1,\ldots ,\widehat{x_j}, \ldots ,x_n]$ is a 
polynomial in all variables but $x_j$. For $f\ne 0$, we 
define the {\it initial $y$-form} denoted $\text{in}_y(f)$ to be 
the sum of all the nonzero terms of $f$ having the highest 
power of $y$, that is, $\text{in}_y(f)=\alpha_d y^d$. For an 
ideal $J\subset R$, define $\text{in}_y(J):= \langle 
\text{in}_y(f) \ | \ f \in J \rangle$. A monomial order $<$ 
on $R$ is said to be {\it $y$-compatible} if it satisfies 
$\text{in}_<(\text{in}_y(f)) = \text{in}_<(f)$ for all 
$f\in R$, where $\text{in}_<(f)$ is the initial term of 
$f$ with respect to $<$. It follows that for 
such an order, we have $\text{in}_<(\text{in}_y(I)) = 
\text{in}_<(I)$ for all ideal $I$. 

Consider an ideal $I$ and a $y$-compatible monomial order 
on $R$. Suppose that $\cG = \{g_1,\ldots ,g_m\}$ is a 
Gr\"obner basis of $I$ with respect to this monomial order, 
and for each $i=1,\ldots , m$, write $g_i = 
y^{d_i}q_i+r_i$, where $y$ does not divide any $q_i$. 
Hence, we have $\text{in}_y(g_i)=y^{d_i}q_i$. It also 
follows that $\text{in}_y(I) = \langle y^{d_i}q_i \ | \ 1\le
i \le m \rangle$. We define the following ideals
\[
C_{y,I} = \langle q_1, \ldots ,q_m \rangle \text{  and  } 
N_{y,I} = \langle q_i \ | \ d_i=0 \rangle.
\]

It is important to observe that the ideals $C_{y,I}$ and $N_{y,I}$ do not depend on the choice of Gr\"{o}bner basis, and in particular do not
depend on the choice of $y$-compatible monomial order.

\begin{defn}
When $\text{in}_y(I) = C_{y,I} \cap (N_{y,I}+\langle y 
\rangle$), we call this a {\it geometric vertex 
decomposition} of $I$ with respect to $y$. We say that the 
geometric vertex decomposition is {\it degenerate} if 
$\sqrt{C_{y,I}} = \sqrt{N_{y,I}}$ or if $C_{y,I}=\langle 1 
\rangle$, and {\it nondegenerate} otherwise.
\end{defn}

Recall that an ideal $I$ is {\it unmixed} if $I$
satisfies $\dim(R/I) = \dim(R/P)$ for all associated 
primes $P \in  \Ass_R(R/I)$. We define the main object of 
study in this paper:

\begin{defn}\label{defn.gvd}
An ideal $I$ of $R=\kk[x_1,\ldots ,x_n]$ is 
\emph{geometrically vertex decomposable} if $I$ is unmixed 
and
    \begin{enumerate}
    \item  $I = \langle 1 \rangle$, or $I$ is generated by 
    a (possibly empty) subset of variables of $R$, or
    
    \item there exists a variable $y=x_j$ of $R$ and a  
    $y$-compatible monomial order such that we have a 
    geometric vertex decomposition $\text{in}_y(I) = 
    C_{y,I} \cap (N_{y,I}+\langle y \rangle$), and the 
    contraction of the ideals $C_{y,I}$ and $N_{y,I}$ to 
    the ring $\kk[x_1,\ldots ,\widehat{x_j}, \ldots ,x_n]$ 
    are geometrically vertex decomposable.
    \end{enumerate}
\end{defn}

Thus, given a geometrically vertex decomposable ideal $I$, one can perform a geometric vertex decomposition with respect to some variable $y$  to obtain a geometrically vertex decomposable $C$-ideal (i.e., $C_{y,I}$) and a geometrically vertex decomposable $N$-ideal (i.e., $N_{y,I}$). Each of these $C$ and $N$ ideals can then be decomposed into their own geometrically vertex decomposable $C$ and $N$ ideals, and so on, until all ideals have the form of item (1) of Definition \ref{defn.gvd}. We refer to such a process of repeatedly performing geometric vertex decompositions, where all ideals at all stages are geometrically vertex decomposable, as a \emph{geometric vertex decomposition process}. We note that an ideal $I$ may have multiple different geometric vertex decomposition processes.

Our main results depend upon a relationship between the 
$h$-polynomials of $R/I$, $R/C_{y,I}$, and
$R/N_{y,I}$.
Recall that the   {\it Hilbert series}
of a graded $R$-module $M= \bigoplus_{i=1}^\infty M_i$ is the generating
function
$$H_M(t) = \sum_{i=0}^\infty (\dim_{\kk} M_i)t^i.$$
By the Hilbert--Serre Theorem \cite[Corollary 4.1.8]{BH}, 
the Hilbert series
can be expressed as  a rational function 
$$H_M(t) = \frac{h_M(t)}
{(1-t)^d}$$ where $h_M(t)$, the {\it $h$-polynomial} of $M$, is a polynomial with integer 
coefficients and $d= \dim M$. 
The following result, which is \cite[Theorem 2.1(e)]{KMY09},
relates the Hilbert series of an ideal to that of its geometric vertex decomposition; for completeness, we have included a proof.

\begin{thm}\label{thm.hilbertseries}
Suppose that $I \subseteq R$ is a homogeneous
ideal that has a geometric vertex 
decomposition with respect to $y$, that is,
${\rm in}_y(I) = C_{y,I} 
\cap (N_{y,I}+\langle y \rangle)$.  Then 
$$H_{R/I}(t) = H_{R/(N_{y,I} + \langle y \rangle)}(t) + tH_{R/C_{y,I}}(t).$$
\end{thm}

\begin{proof}
We have the short exact sequence
\[
0 \longrightarrow \frac{R}{C_{y,I}\cap (N_{y,I}+\langle y \rangle)} \longrightarrow \frac{R}{C_{y,I}} \oplus \frac{R}{N_{y,I}+\langle y \rangle} \longrightarrow \frac{R}{C_{y,I}+N_{y,I}+\langle y \rangle} \longrightarrow 0.
\]
Furthermore, note that $C_{y,I}+N_{y,I}+\langle y \rangle =C_{y,I}+\langle y \rangle$. 
Because ${\rm in}_y(I)$ and $I$ have 
the same Hilbert series (since ${\rm in}_<({\rm in}_y(I)) = {\rm in}_<(I))$, the Hilbert
series of $R/I$ satisfies
\[
H_{R/I}(t) = H_{R/(N_{y,I}+\langle y \rangle)}(t) +
H_{R/C_{y,I}}(t) - H_{R/(C_{y,I}+\langle y \rangle)}(t).
\]
The short exact sequence 
$$ 0 \longrightarrow 
\frac{R}{C_{y,I}}(-1) \stackrel{\times y}{\longrightarrow}
\frac{R}{C_{y,I}} \longrightarrow \frac{R}{C_{y,I}+\langle y \rangle} \longrightarrow 0$$
then implies that
$$H_{R/C_{y,I}}(t) - H_{R/(C_{y,I}+\langle y \rangle)}(t) = tH_{R/C_{y,I}}(t).$$
Consequently, 
$
H_{R/I}(t) = H_{R/(N_{y,I}+\langle y \rangle)}(t) +
tH_{R/C_{y,I}}(t),
$
as desired.
\end{proof}

Theorem \ref{thm.hilbertseries} 
implies a relationship among the $h$-polynomials
$h_{R/I}(t), h_{R/(N_{y,I}+\langle y \rangle)}(t)$ and
$h_{R/C_{y,I}}(t)$.  Note that to explicitly
determine this relationship, one would
need to know the dimension of the three rings. 
As shown below, we know the dimensions 
in the case that the homogeneous ideal $I \subseteq R$
is geometrically vertex decomposable. In fact,
instead of the $h$-polynomial of $R/(N_{y,I}+\langle y \rangle)$,
we use the $h$-polynomial of $R/N_{y,I}$.

\begin{thm}\label{thm.hilbertseries2}
Suppose that $I \subseteq R$ is a homogeneous, geometrically 
vertex decomposable ideal and $\text{in}_y(I) = C_{y,I} 
\cap (N_{y,I}+\langle y \rangle)$ is its geometric vertex 
decomposition with respect to a $y$-compatible monomial order as in Definition \ref{defn.gvd}. If the decomposition is non-degenerate, then
the $h$-polynomial of $R/I$ satisfies
$$h_{R/I}(t) = h_{R/N_{y,I}}(t) + th_{R/C_{y,I}}(t).$$
\end{thm}

\begin{proof}
   % Note by Theorem \ref{thm.hilbertseries}, it suffices
   % to show that $H_{R/N_{y,I}}(t) = H_{R/(N_{y,I}+\langle y %\rangle)}(t)$ under the additional assumption that
   % $I$ is geometrically vertex decomposable and the
   % decomposition is non-degenerate.
    
    Since $I$ is geometrically vertex decomposable, it is 
Cohen-Macaulay \cite[Corollary 4.5]{KR21}. %and hence, Cohen-Macaulay. 
Because the ideals  $C_{y,I}$ and $N_{y,I}$
are also geometrically vertex decomposable, these
ideals are also Cohen-Macaulay.
Since the decomposition $\text{in}_y(I) = 
C_{y,I} \cap (N_{y,I}+\langle y \rangle$) is 
non-degenerate, by \cite[Lemma 2.8]{KR21}, we have 
$\codim(I)=\codim(C_{y,I})=\codim(N_{y,I})+1$. Using these
facts combined with the fact that these 
are ideals are Cohen-Macaulay now gives 
$$\dim(R/I) = \dim(R/C_{y,I})=\dim(R/N_{y,I})-1 
=\dim(R/(N_{y,I}+\langle y \rangle)) =d.$$ Therefore, 
by Theorem \ref{thm.hilbertseries} we have 
\[
\dfrac{h_{R/I}(t)}{(1-t)^d} = \dfrac{h_{R/(N_{y,I}+\langle 
y \rangle)}(t)}{(1-t)^d} + \dfrac{th_{R/C_{y,I}}(t)}{(1-
t)^d} .
%= \dfrac{h_{R/N_{y,I}}(t)}{(1-t)^d} + 
%\dfrac{th_{R/C_{y,I}}(t)}{(1-t)^d}.
\]

To complete the proof, it suffices to show that $h_{R/(N_{y,I}+\langle 
y \rangle)}(t) = h_{R/N_{y,I}}(t).$
Because the generators of $N_{y,I}$ do not involve $y$, we have $R/(N_{y,I}+\langle y \rangle) \cong R'/N_{y,I}$ and $R/N_{y,I} \cong R'/N_{y,I} \otimes\kk[y]$, where $R'=R/\langle y \rangle$. Since $\kk[y]$ is a free module, we get that
\[
H_{R/N_{y,I}}(t) = H_{R'/N_{y,I}}(t)\cdot H_{\kk[y]}(t) \Longleftrightarrow \dfrac{h_{R/N_{y,I}}(t)}{(1-t)^{d+1}} = \dfrac{h_{R'/N_{y,I}}(t)}{(1-t)^{d}}\dfrac{1}{1-t}. 
\]
Comparing the numerators in the expression on the
right hand side, we have $h_{R/N_{y,I}}(t) = h_{R'/N_{y,I}}(t)
= h_{R/(N_{y,I}+\langle y \rangle)}(t)$, completing the proof.
\end{proof}

%%%%%%%%%%%%%%%%%%%%%%%%%%%%%%%%%%%%%%%%%%

\section{Regularity of Geometrically Vertex
Decomposable Ideals}

We consider the 
(Castelnuovo-Mumford) regularity 
of geometically vertex decomposable ideals.
We derive a %very nice 
recursive formula for the  regularity for this family that allows us to compute the regularity of 
various classes of ideals, e.g., Stanley-Reisner
ideals of pure vertex decomposable 
simplicial complexes, and toric ideals 
bipartite graphs in Section \ref{sec.graph}.

The {\it (Castelnuovo-Mumford)
regularity} of a graded $R$-module $M$ is given by
$${\rm reg}(M) = \max\{j-i ~|~ \beta_{i,j}(M) 
\neq 0 \},$$
where $\beta_{i,j}(M)$ denotes the $(i,j)$-th
graded Betti number that appears in the 
minimal graded free resolution of $M$.
The following 
property, which relates the regularity to the 
degree of the $h$-polynomial
in the Hilbert series, shall be of great use.  
 
\begin{lem}[{\cite[Corollary B.28]{Vas}}]\label{lem.reg=deg}
Let $I \subseteq R$ be a homogeneous ideal such that $R/I$ is 
Cohen-Macaulay.  Then ${\rm reg}(R/I) = \deg h_{R/I}(t)$.
\end{lem}
\noindent
Because all geometrically vertex decomposable ideals
are Cohen-Macaulay, for this family of ideals, we can informally
define the regularity of $R/I$ to the be degree of the
$h$-polynomial.

We come to the main result of this section, which 
describes a recursion between regularity values of a 
geometrically vertex decomposable $I$ and the
corresponding ideals $C_{y,I}$ and $N_{y,I}$. The proof 
relies on Theorem \ref{thm.hilbertseries2}.

\begin{thm}
\label{thm.regformula}
Suppose that $I \subseteq R$ is a homogeneous, geometrically 
vertex decomposable ideal and ${\rm in}_y(I) = C_{y,I} 
\cap (N_{y,I}+\langle y \rangle)$ is its geometric vertex 
decomposition with respect to a $y$-compatible monomial order as in Definition \ref{defn.gvd}. If the decomposition is non-degenerate, then
\[
\reg(I)=\reg({\rm in}_y(I))=\max \{\reg(N_{y,I}), 
\reg(C_{y,I})+1 \}. 
\]
Otherwise, if the decomposition is degenerate, then 
$\reg(I)=\reg({\rm in}_y(I))=\reg(C_{y,I})=\reg(N_{y,I})$ if $C_{y,I} \not = \langle 1 \rangle $ and $\reg(I)=\reg({\rm in}_y(I))=\reg(N_{y,I})$ if $C_{y,I} = \langle 1 \rangle$.
\end{thm}

\begin{proof}
Since $I$ is geometrically vertex decomposable, it is 
Cohen-Macaulay \cite[Corollary 4.5]{KR21}. Thus, $\reg(R/I) = \deg h_{R/I}(t)$ 
by Lemma \ref{lem.reg=deg}. Moreover, since $I$ is 
geometrically vertex decomposable, so are  $N_{y,I}$ and
$C_{y,I}$, and thus both of these ideals are Cohen-Macaulay.
Thus $\reg(R/I)= \deg h_{R/I}(t)$,  $\reg(R/C_{y,I})= \deg h_{R/C_{y,I}}(t)$, and 
$\reg(R/N_{y,I})= \deg h_{R/N_{y,I}}(t)$.

If the decomposition is non-degenerate, then by
Theorem \ref{thm.hilbertseries2} we have
\begin{eqnarray*}
\reg(R/I) = \deg h_{R/I}(t) & =& \max\{\deg h_{R/N_{y,I}}(t),
\deg h_{R/C_{y,I}}(t)+1\} \\
& = & \max\{\reg(R/N_{y,I}), \reg(R/C_{y,I})+1\}.
\end{eqnarray*}
Note that the second equality follows from the 
fact the $h$-polynomial of a Cohen-Macaulay ring always
has non-negative coefficients (e.g., see
\cite[Corollary 3.11]{S78}) so there is no cancellation
among the top-degree terms when the polynomials of 
Theomem \ref{thm.hilbertseries2} are added. To recover
the statement of the theorem, use the fact that $\reg(R/I) = \reg(I) -1$ for all
proper homogeneous ideals.   

To show that $\reg(I) = \reg({\rm in}_y(I))$, note that
by \cite[Corollary 4.11]
{KR21}, we have $\text{in}_y(I)$ is also Cohen-Macaulay. 
Since $\text{in}_y(I)$ and $I$ have the same Hilbert 
series (as the Hilbert series of 
$\text{in}_<(\text{in}_y(I))=\text{in}_<(I)$), we get 
\[
\reg(R/I)= \deg h_{R/I}(t) = \deg h_{R/\text{in}_y(I)}(t) = 
\reg(R/\text{in}_y(I)),
\]
as desired.

%Now suppose that the decomposition $\text{in}_y(I) = 
%C_{y,I} \cap (N_{y,I}+\langle y \rangle$) is 
%non-degenerate. By \cite[Lemma 2.8]{KR21}, we have 
%$\codim(I)=\codim(C_{y,I})=\codim(N_{y,I})+1$, and since 
%they are Cohen-Macaulay 
%$\dim(R/I) = \dim(R/C_{y,I})=\dim(R/N_{y,I})-1 
%=\dim(R/(N_{y,I}+\langle y \rangle)) =d$. Therefore, the %recursive formula for Hilbert series in Theorem %\ref{thm.hilbertseries} becomes
%\[
%\dfrac{h_{R/I}(t)}{(1-t)^d} = \dfrac{h_{R/(N_{y,I}+\langle 
%y \rangle)}(t)}{(1-t)^d} + \dfrac{th_{R/C_{y,I}}(t)}{(1-
%t)^d} = \dfrac{h_{R/N_{y,I}}(t)}{(1-t)^d} + 
%\dfrac{th_{R/C_{y,I}}(t)}{(1-t)^d}.
%\]
%Note that in the above equality, $\frac{h_{R/(N_{y,I}+\langle 
%y \rangle)}(t)}{(1-t)^d} = \frac{h_{R/N_{y,I}}(t)}{(1-t)^d}$ holds by the following: since the generators of $N_{y,I}$ do not involve $y$, we have $R/(N_{y,I}+\langle y \rangle) \cong R'/N_{y,I}$ and $R/N_{y,I} \cong R'/N_{y,I} \otimes\kk[y]$, where $R'=R/\langle y \rangle$. Since $\kk[y]$ is a free module, we get that
%\[
%H_{R/N_{y,I}}(t) = H_{R'/N_{y,I}}(t)\cdot H_{\kk[y]}(t) \Longleftrightarrow \dfrac{h_{R/N_{y,I}}(t)}{(1-t)^{d+1}} = \dfrac{h_{R'/N_{y,I}}(t)}{(1-t)^{d}}\dfrac{1}{1-t}, 
%\]
%hence the claim.
%Thus, $\deg h_{R/I}(t) = \max \{\deg h_{R/N_{y,I}}(t), \deg
%h_{R/C_{y,I}}(t) +1 \}$, and the recursive formula for 
%regularity follows in the case of nondegenerate vertex 
%geometric decompostion. 

Now suppose that the decomposition is 
degenerate. If $C_{y,I} = \langle 1 \rangle$, we have $R/I \cong R/N_{y,I}+\langle y \rangle$, thus, the claim follows. Otherwise, if $C_{y,I} \not= \langle 1 \rangle$, the result follows directly from the fact that 
$I = \text{in}_y(I) = C_{y,I} = N_{y,I}$, as shown in 
\cite[Proposition 2.4]{KR21}. 
\end{proof}

The above result recovers the recursive formula for the regularity of the Stanley-Reisner ideal of a 
pure vertex decomposable complex.  Since this is our only
result concerning vertex decomposable
simplicial complexes and Stanley-Reisner ideals, we 
point the reader to \cite[Section 2.1]{KR21}
for notation and terminology that is not explained.
We want to highlight that our result is for
{\it pure} simplicial complexes;  the following 
result can be seen as giving new proofs
for special cases of {\cite[Theorem 4.2]{HW2014} and  
\cite[Corollary 2.11]{MK16}}; in particular,
\cite{HW2014,MK16} do not require the
vertex decomposable simplicial complex to be pure.

\begin{cor}\label{cor.simplicialcomplexcase}
    Let $\Delta$ be a pure vertex decomposable simplicial 
    complex and let $v$ be a shedding vertex of $\Delta$. If 
    $I_\Delta$ is the Stanley-Reisner ideal of 
    $\Delta$, then 
\[
    \reg(R/I_\Delta) = \max \{\reg(R/I_{\Delta_1}), 
    \reg(R/I_{\Delta_2})+1 \},
\]
where $\Delta_1 = {\rm del}_{\Delta}(v)$ is the deletion of $v$, and $\Delta_2 = {\rm lk}_{\Delta}(v)$ is the link of $v$. 
\end{cor}

\begin{proof}
As pointed out in \cite[Proposition 2.9]{KR21}, $I_\Delta$ is geometrically vertex decomposable, and in \cite[Remark 2.5]{KR21} that we have a geometric vertex decomposition with $N_{y,I}+\langle y \rangle=I_{\Delta_1}$ and $C_{y,I}=I_{\text{star}_\Delta(v)}$, where $y$ is the variable corresponding to the vertex $v$. The decomposition is nondegenerate since $v$ is a shedding vertex. Now as $I_{\Delta_2} = I_{\text{star}_\Delta(v)} + \langle y \rangle$, and since $y$ is not in the support of $C_{y,I}$ and $N_{y,I}$, we have $\reg(C_{y,I})= \reg(I_{\Delta_2})$ and $\reg(N_{y,I}) = \reg(I_{\Delta_1})$. Hence, the above formula follows from Theorem \ref{thm.regformula}. 
\end{proof}

\begin{rem}
Under the hypotheses of the above corollary, 
the decomposition is degenerate if and only if $\Delta$ is a cone from $v$ on $\Delta_2$, and in this case, $\reg(I_\Delta)=\reg(I_{\Delta_1})=\reg(I_{\Delta_2})$.
\end{rem}

\begin{rem}
The proof of Corollary \ref{cor.simplicialcomplexcase}, as given
in \cite{HW2014}, uses tools from combinatorial
topology, like the Mayer–Vietoris sequence.
On the other hand, Corollary \ref{cor.simplicialcomplexcase} is proved
in \cite{MK16} by first
computing the projective dimension of $I_\Delta^\vee$,
the corresponding Alexander dual of $I_\Delta$,
and then using the fact that this value equals ${\rm reg}(R/I_\Delta)$.  Our proof of
Corollary \ref{cor.simplicialcomplexcase} provides
an entirely new approach in the case
of pure vertex decomposable simplicial complexes.
\end{rem}

\begin{ex}
\label{ex.gvdideal}  
    The ideal $I=\langle y(zs-x^2),ywr,wr(z^2+zx+wr+s^2) \rangle$ is geometrically vertex decomposable with $C_{y,I} = \langle zs-x^2,wr \rangle $ and $N_{y,I}= \langle wr(z^2+zx+wr+s^2) \rangle$, and the geometric vertex decomposition is nondegenerate, see \cite[Example 2.16]{KR21}. The ideal $N_{y,I}$ is generated by one polynomial of degree $4$ so $\reg(N_{y,I})=4$. For the ideal $C_{y,I}$, since its two generators have separate variables, we have 
    \[
    \reg(C_{y,I}) = \reg(\langle zs-x^2 \rangle) +\reg( \langle wr \rangle)-1 = 2+2-1=3.
    \]
    Therefore, by Theorem \ref{thm.regformula}, $\reg(I)=4$.
\end{ex}

\begin{rem}
\label{rem.weakgvd}
As defined in \cite[Definition 4.6]{KR21}, an ideal $I\subseteq R$ is called {\it weakly geometrically vertex decomposable} if $I$ is unmixed and if any
of the following is true:
\begin{enumerate}
    \item  $I = \langle 1 \rangle$, or $I$ is generated by a (possibly empty) subset of variables of $R$, or
    \item for some variable $y=x_j$ of $R$, there is a \emph{degenerate} geometric vertex decomposition $\text{in}_y(I) = C_{y,I} \cap (N_{y,I}+\langle y \rangle$), and the contraction of $N_{y,I}$ to the ring $\kk[x_1,\ldots ,\widehat{x_j}, \ldots ,x_n]$ is weakly geometrically vertex decomposable; or
    \item for some variable $y=x_j$ of $R$, there is a \emph{nondegenerate} geometric vertex decomposition $\text{in}_y(I) = C_{y,I} \cap (N_{y,I}+\langle y \rangle$), the contraction of $C_{y,I}$ to the ring $\kk[x_1,\ldots ,\widehat{x_j}, \ldots ,x_n]$ is geometrically vertex decomposable, and $N_{y,I}$ is radical Cohen-Macaulay.
\end{enumerate}

The proofs of Theorems \ref{thm.hilbertseries2} and \ref{thm.regformula} can be adapted easily to the weakly geometrically vertex decomposable ideals.  In these proofs, we only used the geometrically vertex decomposable property of $I$ to obtain that the ideals $I, \text{in}_y(I), C_{y,I}$, and $N_{y,I}$ are Cohen-Macaulay, which is true by \cite[Corollary 4.8, 4.11]{KR21} without the geometrically vertex decomposable assumption. Furthermore, the height lemma \cite[Lemma 2.8]{KR21} and the fact that $I$ is radical \cite[Corollary 4.8]{KR21} are also true in weakly geometrically vertex decomposable setting. Therefore, we have 
the same recursive formula for weakly geometrically vertex decomposable ideals.  
\end{rem}

\begin{ex}
\label{ex.weakgvd} 
    The ideal $I=\langle y(zs-x^2),ywr,wr(x^2+z^2+wr+s^2) \rangle$ is weakly geometrically vertex decomposable with $C_{y,I} = \langle zs-x^2,wr \rangle $ and $N_{y,I}= \langle wr(x^2+z^2+wr+s^2) \rangle$ but $I$ is not geometrically vertex decomposable, see \cite[Example 4.10]{KR21}. Nevertheless, by Remark \ref{rem.weakgvd}, we still have 
    \[
    \reg(I) = \max \{\reg(N_{y,I}), \reg(C_{y,I})+1 \} = \max \{ 4,3+1\} =4.
    \]
\end{ex}

\begin{rem}
\label{rem.IandNCM} 
More generally, the proofs of Theorems \ref{thm.hilbertseries2}
and \ref{thm.regformula} can be adapted easily to the case when we only require $I$ is a homogeneous ideal that possesses a geometric vertex decomposition and the ideals $I$ and $N_{y,I}$ are Cohen-Macaulay, as in this proof, we only needed that the ideals $I, \text{in}_y(I), C_{y,I}$, and $N_{y,I}$ are Cohen-Macaulay, which is true by \cite[Corollary 4.8, 4.11]{KR21}, and the height lemma \cite[Lemma 2.8]{KR21} is also true when $I$ is Cohen-Macaulay and the geometric vertex decomposition is nondegenerate. If in addition, $I$ is radical then the formula in degenerate case also works by the same argument.
\end{rem}

\begin{ex}
\label{ex.nongvd}
Consider the ideal $I= \langle yz - xw, xy 
    \rangle$. One can check that $I$ is not geometrically vertex decomposable. Nevertheless, using the lexicographical order $x>y>z>w$, the Gr\"obner basis of $I$ is $\{yz - xw, xy,y^2z \}$. One can check that $I$ has a nondegenerate geometric vertex decomposition with $C_{x,I} = \langle y, w \rangle $ and $N_{x,I}= \langle y^2z \rangle$. Since $I$ and $N_{x,I}$ are Cohen-Macaulay, we get
    $\reg(R/I) = \max \{\reg(R/N_{x,I}), \reg(R/C_{x,I})+1 \} = 2.$
\end{ex}

%%%%%%%%%%%%%%%%%%%%%%%%%%%%%%%%%%%%%%%%%%%%%%
\section{Multiplicity of  
Geometrically Vertex Decomosable Ideals}
\label{sec.otherinvar}

This short section considers the 
the multiplicity of (weakly) geometrically 
    vertex decomposable ideals. As with Theorem \ref{thm.regformula}, our results rely on using Theorem \ref{thm.hilbertseries2}
    to relate the $h$-polynomial of $I$ to 
    the $h$-polynomials of $C_{y,I}$ and $N_{y,I}$. 

\begin{defn}
    Let $M$ be an $R$-module with Hilbert series
    $H_M(t) = \frac{h_M(t)}{(1-t)^d}$ where $d = \dim M$.
    Then the {\it multiplicity} of $M$ is $e(M) = 
    h_M(1)$.
\end{defn}

The leading coefficient of the \emph{Hilbert polynomial} of $M$ is given by $\frac{e(M)}{d!}$, and when $M=R/I$ with $I$ is the defining ideal of a projective variety, then $e(M)$ is the degree of the variety.

\begin{thm}\label{thm.eformula}
    Suppose that $I \subseteq R$ is a homogeneous, (weakly) geometrically 
    vertex decomposable ideal and ${\rm in}_y(I) = C_{y,I} 
    \cap (N_{y,I}+\langle y \rangle)$ is its geometric vertex 
    decomposition with respect to a $y$-compatible monomial order as in Definition \ref{defn.gvd}. If the decomposition is non-degenerate, then
    \[
    e(R/I) = e(R/N_{y,I}) + e(R/C_{y,I}). 
    \]
    Otherwise, if the decomposition is degenerate, we have $e(R/I) = e(R/N_{y,I}) = e(R/C_{y,I})$ if $C_{y,I}\neq \langle 1 \rangle$ and $e(R/I) = e(R/N_{y,I})$ if $C_{y,I} = \langle 1 \rangle$.
\end{thm}

\begin{proof}
    By Theorem \ref{thm.hilbertseries2}, 
    if the decomposition is non-degenerate we have 
    $h_{R/I}(t) = h_{R/N_{y,I}}(t) + th_{R/C_{y,I}}(t).$ Evaluating at $t=1$ now gives the result. The result in the degenerate case again follows from the fact $R/I \cong R/(N_{y,I}+\langle y \rangle)$ if $C_{y,I} = \langle 1 \rangle$, and %$\sqrt{\text{in}_y(I)} = \sqrt{C_{y,I}} = \sqrt{N_{y,I}}$ if 
$I = \text{in}_y(I) = C_{y,I} = N_{y,I}$ (by
\cite[Proposition 2.4]{KR21}) if  
    $C_{y,I} \not= \langle 1 \rangle$. 
\end{proof}    
    %{\bf AVT: same comment as 
    %the proof in previous section;  In this case, I can
    %see that if they define the same variety, they must have
    %the same hilbert polynomial, and so the same degree. Is
   % that the argument?}

\begin{ex}
    \label{ex.mult}
    \begin{enumerate}
        \item Referring to Example \ref{ex.gvdideal}, $I=\langle y(zs-x^2),ywr,wr(z^2+zx+wr+s^2) \rangle$ is geometrically vertex decomposable with $C_{y,I} = \langle zs-x^2,wr \rangle $ and $N_{y,I}= \langle wr(z^2+zx+wr+s^2) \rangle$. Hence, $e(R/I) = e(R/N_{y,I}) + e(R/C_{y,I}) = 4+4 =8.$
        \item Referring to Example \ref{ex.nongvd}, $I= \langle yz - xw, xy 
    \rangle$ has a nondegenerate geometric vertex 
    decomposition with $C_{x,I} = \langle y, w \rangle $ and $N_{x,I}= \langle y^2z \rangle$. Since $I$ and $N$ are Cohen-Macaulay, the argument as in Remark \ref{rem.IandNCM} applies. Hence, $e(R/I) = e(R/N_{x,I}) + e(R/C_{x,I}) = 3+1=4$.
    \end{enumerate}
\end{ex}
%%%%%%%%%%%%%%%%%%%%%%%%%%%%%%%%%%%%%%%%%%%

\section{The $a$-invariant and the Hilbertian property}

%{\bf AVT: This section is very rough. 
%I just added some definitions and
%basic results}.
In this section, we study the $a$-invariant, as well as the related \emph{Hilbertian property}, of a geometrically vertex decomposable ideal. 
As in Theorems \ref{thm.regformula} and \ref{thm.eformula}, our results rely on relating
the $h$-polynomial of $I$ to 
the $h$-polynomials of $C_{y,I}$ and
$N_{y,I}$. 

We begin by recalling the definition of the $a$-invariant.  %We first consider the $a$-invariant:

\begin{defn}
    Let $M$ be an $R$-module with Hilbert series
    $H_M(t) = \frac{h_M(t)}{(1-t)^d}$ where 
    $d = \dim M$.
    Then the {\it $a$-invariant} of $M$ is $a(M) = 
    \deg h_M(t) - d$, that is, the degree
    of $H_M(t)$ as a rational polynomial.
\end{defn}

Among other things, the $a$-invariant measures the  top 
degree of the  local cohomology module 
$H_{\mathfrak{m}}^{\dim M}(M)$.  In 
addition, the $a$-invariant defines
the following property for Cohen-Macaulay ideals.

\begin{comment}
\begin{defn}\label{dfn.hilbertian}
    A homogeneous ideal $I \subseteq R$ is {\it Hilbertian}
    if $a(R/I) < 0$. The ideal $I$ is \emph{almost Hilbertian} if $a(R/I) \le 0$.
\end{defn}
\noindent
Note that if $I$ is Hilbertian, then
the Hilbert function $HF(t)$ and
the Hilbert polynomial $HP(t)$ of $R/I$
agree for all $t \geq 0$ (this property 
is sometimes taken as the definition of Hilbertian, e.g., see \cite{AK1989,SY23}). If the ideal is almost Hilbertian, then 
the Hilbert function
$HF(t)$
and the Hilbert polynomial $HP(t)$ of $R/I$ may
only disagree at the value of $t=0$.
\end{comment}

\begin{defn}\label{dfn.hilbertian}
    A homogeneous ideal $I \subseteq R$ is {\it Hilbertian} if 
 $HP_{R/I}(t) = HF_{R/I}(t)$ for all $t \geq 0$. The ideal $I$ is \emph{almost Hilbertian} if $HP_{R/I}(t) = HF_{R/I}(t)$ for all $t \geq 1$.
\end{defn}
By Serre's formula, see, for instance, \cite[Theorem 4.4.3]{BH}, when $R/I$ is Cohen-Macaulay, since the difference between the Hilbert function and the Hilbert polynomial of $R/I$ is given by the dimension of its top degree local cohomology, the Hilbertian property can be defined by means of the $a$-invariant as follows.
\begin{lem}
    \label{lem.hilbertian}
    A homogeneous, Cohen-Macaulay ideal $I \subseteq R$ is Hilbertian
    if and only if $a(R/I) < 0$. The ideal $I$ is almost Hilbertian if $a(R/I) \le 0$.
\end{lem}

We now have a result similar to 
Theorems \ref{thm.regformula} and \ref{thm.eformula}:

\begin{thm}
\label{thm.aformula}
Suppose that $I \subseteq R$ is a homogeneous, (weakly) geometrically 
vertex decomposable ideal and ${\rm in}_y(I) = C_{y,I} 
\cap (N_{y,I}+\langle y \rangle)$ is its geometric vertex 
decomposition with respect to a $y$-compatible monomial order as in Definition \ref{defn.gvd}. If the decomposition is non-degenerate, then
\[
a(R/I) =\max \{a(R/N_{y,I})+1, 
a(R/C_{y,I})+1 \}. 
\]
Otherwise, if the decomposition is degenerate, we have $a(R/I) = a(R/N_{y,I}) = a(R/C_{y,I})$ if $C_{y,I}\neq \langle 1 \rangle$ and $a(R/I) = a(R/N_{y,I})$ if $C_{y,I} = \langle 1 \rangle$.
\end{thm}

\begin{proof}
As noted in the proof of Theorem \ref{thm.regformula}, the ideals $I$, $C_{y,I}$
and $N_{y,I}$ are all Cohen-Macaulay.  
Consequently, by Lemma \ref{lem.reg=deg},
${\rm reg}(R/I) = a(R/I) + \dim(R/I)$, 
${\rm reg}(R/N_{y,I}) = a(R/N_{y,I}) +
\dim(R/N_{y,I})$ and ${\rm reg}(R/C_{y,I})
= a(R/C_{y,I})+ \dim(R/C_{y,I}).$
If the decomposition is non-degenerate,
then when applying Theorem \ref{thm.regformula} then gives
\small
$$a(R/I) + \dim(R/I) = 
\max\{a(R/N_{y,I}) + \dim(R/N_{y,I}), a(R/C_{y,I})+ \dim(R/C_{y,I})+1)\}.$$
\normalsize  
We have 
$\dim(R/I) = \dim(R/C_{y,I}) = \dim(R/N_{y,I})-1$,
as shown in the proof of Theorem \ref{thm.regformula}.
The conclusion now follows.

The degenerate case follows by same reasoning as in the proof of Theorem \ref{thm.regformula}.
\end{proof}

\begin{rem}\label{rem.aformula}
In Theorem \ref{thm.aformula}, we are viewing
$C_{y,I}$ and $N_{y,I}$ as ideals of $R$.  But by
Definition \ref{defn.gvd}, we can also view
these ideals as ideals of $R' = R/\langle y \rangle$.
Since no generator of $N_{y,I}$, respectively $C_{y,I}$, is divisible by $y$, we have $\dim(R/N_{y,I}) -1 = \dim(R'/N_{y,I})$,
respectively
$\dim(R/C_{y,I}) - 1 = \dim(R'/C_{y,I})$. Theorem \ref{thm.aformula} thus implies that $
a(R/I) =\max \{a(R'/N_{y,I}), 
a(R'/C_{y,I}) \}$ if the geometric vertex decomposition $\text{in}_y(I) = C_{y,I}\cap (N_{y,I}+\langle y\rangle)$ is non-degenerate. If the decomposition is degenerate, then  $a(R/I) = a(R'/N_{y,I})-1$.
\end{rem}

\begin{rem}
\label{rem.regeitherCN}
    We remark that the $a$-invariant of a geometrically vertex decomposable ideal $I$ can equal to either that of the $C$ ideal or the $N$ ideal of its decomposition. This is, by Lemma \ref{lem.reg=deg}, equivalent to saying that the regularity of a geometrically vertex decomposable ideal $I$ can be either $\reg(C)+1$ or $\reg(N)$ in the formula of Theorem \ref{thm.regformula}. We refer the readers to Example \ref{ex.CNFerrers} for concrete examples using the results of Theorem \ref{thm.Ferrer} in Section \ref{sec.graph}. 
\end{rem}

Using Theorem \ref{thm.aformula} and Remark \ref{rem.aformula}, we can easily prove that the $a$-invariant of a geometrically vertex decomposable ideal is always non-positive.  In particular, geometrically
vertex decomposable ideals are almost
Hilbertian. 

\begin{cor}\label{cor.nonpos}
Let $I\subset R$ be a proper homogeneous geometrically vertex decomposable ideal. Then  $a(R/I)\leq 0$. In particular, $I$ is almost Hilbertian.
\end{cor}

\begin{proof}
We induct on the number of variables in $R = \mathbb{K}[x_1,...,x_n].$ 
If $n = 0$, the result is trivial. 
If $n = 1$, then the only proper homogeneous geometrically vertex decomposable ideals are $\langle x_1\rangle$ and $\langle 0\rangle$. Thus,  the  result holds.

More generally, consider a homogeneous geometrically vertex decomposable ideal $I\subseteq R = \mathbb{K}[x_1,\dots, x_n]$ with $n\geq 2$. Then there is a variable $y = x_j$ and a geometric vertex decomposition $\text{in}_y(I) = C_{y,I} \cap (N_{y,I}+\langle y\rangle )$. If this is a degenerate geometric vertex decomposition, then by Remark \ref{rem.aformula} and the induction hypothesis,  we have %Theorem \ref{thm.aformula}, we have $a(R/I) = a(R/N_{y,I})$. 
%Since $N_{y,I}$ has a generating set that doesn't involve $y$, we have $a(R/N_{y,I}) = a(R'/N_{y,I})-1$. Therefore,
$a(R/I) = a(R'/N_{y,I})-1<0.$
%where the rightmost inequality holds by the induction hypothesis.
If $\text{in}_y(I) = C_{y,I} \cap (N_{y,I}+\langle y\rangle)$ is a non-degenerate geometric vertex decomposition then, by Remark \ref{rem.aformula}, we have that 
$a(R/I) = \max\{a(R'/N_{y,I}), a(R'/C_{y,I})\}$. Hence the desired result follows from the induction hypothesis.
\end{proof}
%If N and C are saturated, then both the $a$-invariants of R'/N and R'/C are negative by induction. Thus so is the $a$-invariant of R/I and we are done.
%So, it suffices to argue that both N and C must be saturated. Since gvd ideals are radical, the only way for N or C to not be saturated is for either N or C to be the irrelevant ideal in R'.
%But, ht I = ht C = ht N+1. So N is not the maximal ideal. Suppose that C is. Then C = <x_2,.., x_n>
%So, <yx_2,.., yx_n>\subseteq $\in_y I$
%\end{proof}

%The following is immediate. % from the fact that the 
Inspired by the definition of Hilbertian ideals, we next study geometrically vertex decomposable ideals for which the $a$-invariant is always negative. We single out the following class first, as the proof is straightforward.

\begin{cor}
Let $I\subset R$ be a proper homogeneous geometrically vertex decomposable ideal, and suppose that there is minimal generating set of $I$ that does not involve all the variables in $R$. Then  $a(R/I)< 0$. In particular, $I$ is Hilbertian. 
\end{cor}

\begin{proof}
Suppose that there is some minimal generating set of $I$ which does not involve the variable $y = x_i$. Then the geometric vertex decomposition $\text{in}_y(I) = C_{y,I} \cap (N_{y,I}+\langle y\rangle )$ is degenerate.  Thus, by Remark \ref{rem.aformula} and Corollary \ref{cor.nonpos}, 
$a(R/I) = a(R'/N_{y,I})-1<0.$
\end{proof}

We now consider a larger sub-class of geometrically vertex decomposable ideals for which the $a$-invariant is always negative. 

%Let $I$ be a homogeneous, saturated, geometrically vertex decomposable ideal. We say that $I$ is \emph{$C$-saturated} if the geometrically vertex decomposable property of $I$ can be realized by a fixed sequence of geometric vertex decompositions for which none

%Suppose that $I$ is a homogeneous ideal of $R = \mathbb{K}[x_1,\dots, x_n]$ and $\text{in}_y I = C_{y,I}\cap (N_{y,I}+\langle y\rangle)$ is a non-degenerate geometric vertex decomposition. We say that this geometric vertex decomposition is \emph{$C$-saturated} if the contraction of  $C_{y,I}$ to $\mathbb{K}[x_1,\dots, \hat{y},\dots, x_n]$ is not the homogeneous maximal ideal $\langle x_1,\dots,\hat{y},\dots, x_n\rangle$.

%\begin{defn}
%    An ideal $I$ of $R = \mathbb{K}[x_1,\dots, x_n]$ is \emph{$S$-geometrically vertex decomposable} if 
%\end{defn}

 %  In the above definition, $S$ stands for saturated. 

%We now define the class of geometrically vertex decomposable ideals which will have negative $a$-invariant (see Proposition \ref{prop:negainv} below).

\begin{defn}\label{defn.C-satgvd}
A proper homogeneous ideal $I$ of $R=\kk[x_1,\ldots ,x_n]$ is 
\emph{$C$-saturated geometrically vertex decomposable} if $I$ is saturated and unmixed 
and
    \begin{enumerate}
    \item  $I$ is generated by 
    a (possibly empty) subset of variables of $R$, or
   \item there exists a variable $y=x_j$ of $R$ and a  
    $y$-compatible monomial order such that we have a degenerate geometric vertex decomposition $\text{in}_y(I) = 
    C_{y,I} \cap (N_{y,I}+\langle y \rangle$) where $N_{y,I}$ is $C$-saturated geometrically vertex decomposable, or
    \item there exists a variable $y=x_j$ of $R$ and a  
    $y$-compatible monomial order such that we have a %$C$-saturated  
    (non-degenerate) geometric vertex decomposition $\text{in}_y(I) = 
    C_{y,I} \cap (N_{y,I}+\langle y \rangle$), and the 
    contraction of the ideals $C_{y,I}$ and $N_{y,I}$ to 
    the ring $\kk[x_1,\ldots ,\widehat{x_j}, \ldots ,x_n]$ 
    are $C$-saturated geometrically vertex decomposable.
    \end{enumerate}
\end{defn}

%\textcolor{blue}{Do we have something to tie the definition of $C$-saturated ideals to the $C$-saturated geometric vertex decomposition? Eg, is it true that an ideal $I$ is $C$-saturated gvd if and only if all non-degenerate decompositions in the process are $C$-saturated? Also, is it true that it is equivalent to the fact that all the $C$ ideals are not the irrelevant ideals?

\begin{rem}
    \label{rem.C-saturated}
    Let $I$ be a proper, saturated, homogeneous geometrically vertex decomposable ideal. 
    We will now check that $I$ is $C$-saturated geometrically vertex decomposable if and only if there is some geometric vertex decomposition process of $I$ in which every $C$-ideal that appears in this decomposition process is \emph{not} an irrelevant ideal (contracted to its appropriate polynomial ring). The forward direction is immediate by Definition \ref{defn.C-satgvd}.

    %First suppose that $I$ is $C$-saturated geometrically vertex decomposable. Then, by Definition \ref{defn.C-satgvd}, there is some decomposition process for $I$ in which no proper $C$-ideal is an irrelevant ideal (contracted to its appropriate polynomial ring). 
    
    Conversely, suppose that there is some decomposition process for $I$ in which no $C$-ideal is the irrelevant ideal (in its appropriate polynomial ring). To verify that $I$ is $C$-saturated geometrically vertex decomposable, it  suffices to check that every $N$ ideal that appears in the given decomposition process is saturated. Furthermore, 
    since $N$ is geometrically vertex decomposable, and hence radical, this is equivalent to checking that each $N$ ideal is not the irrelevant ideal. So, consider a geometric vertex decomposition $\text{in}_{x_i}(J) = C\cap (N+\langle x_i\rangle)$ in the given decomposition process. We will show that if $J$ and $C$ are saturated, then so is $N$.

   % So, suppose that
    %there is some proper ideal
    %So, suppose that $J\supseteq I$ appears in the decomposition process where $J$ is not the irrelevant ideal, there is a geometric vertex decomposition $\text{in}_{x_i}(J) = C\cap (N+\langle x_i\rangle)$ for an appropriately chosen variable $x_i$, and $N$ \emph{is} the irrelevant ideal in its appropriate polynomial ring.

   If the geometric vertex decomposition $\text{in}_{x_i}(J) = C\cap (N+\langle x_i\rangle)$ is non-degenerate, then since $\codim I = \codim C = \codim N+1$, the contraction of $N$ cannot be the irrelevant ideal as the contraction of $C$ is not. If the geometric vertex decomposition is degenerate then either $J = \text{in}_{x_i}(J) = C = N$, or $J$ contains a linear form $cx_{i}+r$, where $c\in \mathbb{K}$ is a non-zero constant. 
In the first case, since $C$ is not the irrelevant ideal, neither is $N$. In the second case, if $N$ were the irrelevant ideal (in its appropriate polynomial ring), then $J$ would be the irrelevant ideal in its ring; indeed, all generators of $N$ are contained in $J$ and $J$ contains the linear form $cx_i+r$. Since $J$ is not the irrelevant ideal in its ring, neither is $N$.
\end{rem}

\begin{ex}
\label{ex.nonC-saturated}
    Consider the homogeneous and unmixed ideal $I = \langle yz, x+z\rangle \subseteq \mathbb{K}[y,x,z]$. The given generators are a Gr\"obner basis for the $y$-compatible monomial order Lex with $y>x>z$. We have the geometric vertex decomposition $\text{in}_y (I) = C_{y,I}\cap (N_{y,I}+\langle y \rangle)$, where $C_{y,I} = \langle z,x\rangle$ and $N_{y,I} = \langle x+z\rangle$. Observe that $C_{y,I}$ is the irrelevant ideal in $\mathbb{K}[x, z]$.
    Alternatively, there are geometric vertex decompositions $\text{in}_x(I) = C_{x,I}\cap (N_{x,I}+\langle x \rangle)$ or 
    $\text{in}_z(I) = C_{z,I}\cap (N_{z,I}+\langle z \rangle)$. In each case, 
   $C_{x,I}$ or $C_{z,I}$ are irrelevant ideals (in their respective polynomial rings $\mathbb{K}[y,z]$ and $\mathbb{K}[x,y]$).
    Hence, for each potential decomposition process of $I$, one encounters irrelevant ideals. Thus, $I$ is not $C$-saturated geometrically vertex decomposable. %this non-degenerate geometric vertex decomposition is not $C$-saturated. %Thus, while $I$ is geometrically vertex decomposable, it is not 
\end{ex}

%{\color{red}I'm not sure I'm understanding what the next example is doing. Is it needed?}
%\begin{ex}
%    Consider the ideal $J= \langle ab, ay, yz, x+z\rangle \subseteq \mathbb{K}[a,y,x,z,w]$. The given generators are a Gr\"obner basis for the $a$-compatible monomial order Lex with $a>y>x>z>w$. One can check that $J$ is geometrically vertex decomposable and we have a decomposition with $C_1:=C_{a,J} = \langle y,w,xz \rangle$ and $N_1:=N_{a,J} = \langle yz,x+z \rangle$. Furthermore, while $C_1$ is $C$-saturated as it has a degenerate geometric vertex decomposition with $C_{a,C_1}=\langle 1 \rangle$ and $N_{a,C_1}=\langle w,y+z\rangle$, the ideal $N_1:=N_{a,J}$ is not $C$-saturated by Example \ref{ex.nonC-saturated}. Thus, $I$ is not $C$-saturated. This shows that in order to know a geometrically vertex decomposable ideal is $C$-saturated, one may need to check all contraction of the $C$-ideals in the appropriate polynomial rings in the geometric vertex decomposition process.
%\end{ex}

The next result shows that $C$-saturated geometrically vertex decomposable ideals are always Hilbertian.

\begin{prop}
\label{prop:negainv}
  
 Let $I$ be a proper homogeneous geometrically vertex decomposable ideal in a polynomial ring $R = \mathbb{K}[x_1,\dots, x_n]$ with $n\geq 1$. %Then $a(R/I)<0$ if and only 
    If $I$ is $C$-saturated geometrically vertex decomposable, then $a(R/I)<0$. In particular, $I$ is Hilbertian.
\end{prop}

\begin{proof}
    Our argument is nearly identical to the proof of Theorem \ref{cor.nonpos}.
    We induct on the number of variables in $R = \mathbb{K}[x_1,...,x_n].$ 
If $n = 1$, the only $C$-saturated geometrically vertex decomposable ideal is $\langle 0\rangle$ and the  result holds.

More generally, consider a homogeneous, $C$-saturated geometrically vertex decomposable ideal $I\subseteq R = \mathbb{K}[x_1,\dots, x_n]$ with $n\geq 2$. Then there is a variable $y = x_j$ and a geometric vertex decomposition $\text{in}_y(I) = C_{y,I} \cap (N_{y,I}+\langle y\rangle )$. If this is a degenerate geometric vertex decomposition then, by Remark \ref{rem.aformula} and the induction hypothesis,  we have 
$a(R/I) = a(R'/N_{y,I})-1<0$ as $N_{y,I}\subseteq R'$ is $C$-saturated geometrically vertex decomposable.

%where the rightmost inequality holds by the induction hypothesis. 
If $\text{in}_y(I) = C_{y,I} \cap (N_{y,I}+\langle y\rangle)$ is a non-degenerate $C$-saturated geometric vertex decomposition then, by Remark \ref{rem.aformula}, we have that 
$a(R/I) = \text{max}(a(R'/N_{y,I}), a(R'/C_{y,I}))$. Hence the desired result follows from the induction hypothesis.
\end{proof}

%{\color{red}
%Recall the gvd tree. If the $C$ ideals are always saturated then the $a$-invariant is always negative $a(R/I)<0$. Add this here? Note that by height considerations, $N$ is always saturated. it is just $C$ that may not be.

%What about the vertex decomposable case? $C$ saturated will hold for any vertex decomposition that is not of a disconnected zero dimensional complex. **I think subword complexes have this property. This would answer the question in Ada's paper about type $A$ quiver ideals, as well as homogeneous KL ideals**
%}
%{\color{blue}Thai: Yes I think with that reasoning, the S-R ideal of any (pure?) vertex decomposable simplicial complex is Hilbertian}.

\begin{rem}
    A Stanley-Reisner ideal of a
    vertex decomposable simplicial complex is Hilbertian if and only if in its vertex decomposition process, at every step (each step correspond to removing one vertex from the vertex set) except the last one (where all simplicial complexes are at most one point), taking the link gives all non-empty simplicial complexes. Note that a connected vertex decomposable simplicial complex can have empty links in its vertex decomposition process. For example, consider the graph $C_3$, the three cycle. The link of any vertex is a simplicial complex consists of two disconnected points, and the link of this simplicial complex is empty. Its Stanley-Reisner ideal $ I= \langle xyz \rangle$ is vertex decomposable, but $R/I$ is not Hilbertian. 
\end{rem}

We provide examples of $C$-saturated ideals. 
\begin{ex}
    \label{ex.determinantal}
    Given a permutation $w\in S_n$, there is an associated generalized determinantal ideal $I_w \in \kk[x_{ij}, 1\le i,j\le n]$ called a Schubert determinantal ideal. Then, there exists a lexicographical monomial order $<$ such that $\text{in}_<(I_w)$ is the Stanley-Reisner ideal of an antidiagonal complex (see \cite[Section 5]{KR21} and \cite[Chapter 16]{MSbook}). Moreover, the complex is $<$-compatibly vertex decomposable, that is, the Stanley-Reisner ideal of its link and deletion are the initial ideals of the $C$- and the $N$-ideal in the geometric vertex decomposition of $I_w$, respectively. One can check that this antidiagonal complex has no empty links in its decomposition process using the description given in \cite[Theorem 16.43]{MSbook}. Therefore, $\text{in}_<(I_w)$ and $I_w$ are $C$-saturated geometrically vertex decomposable, and thus, are Hilbertian. The same argument applies to recover the result that for a Kazhdan-Lusztig variety $\mathcal{N}_{v,w}$ defined by homogeneous equations, its homogeneous coordinate ring is Hilbertian if and only if $v \ne w$, as  proved in \cite[Theorem 5.3]{SY23}.
\end{ex}

\section{Applications to toric ideals of graphs}
\label{sec.graph}

In this section we apply 
Theorems \ref{thm.regformula}, \ref{thm.eformula} and
\ref{thm.aformula} to study
the invariants of toric ideals of (bipartite) graphs.
By leveraging the result that the toric ideals of
bipartite graphs are geometrically vertex decomposable
(see \cite[Theorem 5.8]{CSRV22}), we 
can give new proofs for a number of known results 
(e.g. \cite{ADS22,CN09,FHKV21}) using our techniques.

\subsection{Background on toric ideals of graphs}
We begin with the relevant background on toric
ideals of graphs.  Let $G = (V,E)$ be a finite simple
graph with vertex set $V = \{x_1,\ldots,x_n\}$ and edge
set $E = \{e_1,\ldots,e_q\}$.  If we need
to highlight the graph, we sometimes write
$V(G)$ and $E(G)$ for the vertices and edges of $G$.
Abusing notation, we let 
the $x_i$'s and $e_j$'s also denote variables,
and let $\kk[E] = \kk[e_1,\ldots,e_q]$ 
and $\kk[V] = \kk[x_1,\ldots,x_n]$. We define
a $\kk$-algebra homomorphism $\varphi: \kk[E] 
\rightarrow \kk[V]$ by $\varphi(e_i) = x_jx_k$ where
$e_i = \{x_j,x_k\} \in E$.  The kernel of $\varphi$, denoted $I_G$,
is the {\it toric ideal of $G$}.  

The ideal $I_G$ is a toric ideal because 
it is a prime binomial ideal;  for this fact and for more details about $I_G$, see \cite[Chapter 5]{HHO}
or \cite[Chapter 10]{V}.  A (non-minimal) set of
generators of $I_G$ can be described in terms of 
closed even walks of the graph.  A sequence 
of distinct edges 
$\Gamma = (e_{i_1},e_{i_2},\ldots,e_{i_t})$ is a {\it walk} 
if $e_{i_j} \cap e_{i_{j+1}}\neq \emptyset$ 
for $1 \leq j \leq t-1$.  The
walk is {\it closed} 
if $e_{i_t} \cap e_{i_1} \neq \emptyset$.  The
walk is {\it even} if $t$ is even.  
A closed walk $\Gamma$ is a
{\it cycle} if no edges in $\Gamma$ are repeated.
We can associate with every
closed even walk $\Gamma = (e_{i_1},\ldots,e_{i_{2m}})$ a binomial of the form:
$$f_\Gamma = e_{i_1}e_{i_3}\cdots e_{i_{2m-1}}-e_{i_2}e_{i_4}
\cdots e_{i_{2m}}.$$

Recall that a graph
$G$ is {\it bipartite} if the vertex set
$V$ can be partitioned into two disjoint sets
$V = V_1 \cup V_2$ such that every edge $e \in E$
statisfies $e\cap V_1 \neq \emptyset$ and $e \cap V_2 \neq
\emptyset$.  The next result now gives a set of
generators for toric ideals of (bipartite) graphs.
For integers $m,n \geq 1$, 
the {\it complete bipartite graph} $K_{m,n}$ is the 
graph with vertex set $V = \{x_1,\ldots,x_m\} \cup \{y_1,\ldots,y_n\}$ and edge set $E = \{\{x_i,y_j\} ~|~ 1 \leq i \leq m, ~
1 \leq j \leq n\}$.

 \begin{thm}[{\cite[Proposition 3.1]{Vart}}]
\label{thm.toricgenerators}
If $G$ is a finite simple graph with toric ideal
$I_G$, then  
$$I_G = \langle f_\Gamma ~|~ \mbox{$\Gamma$ is
a closed even walk of $G$}\rangle.$$
In addition, if $G$ is bipartite, then $I_G =
 \langle f_\Gamma ~|~ \mbox{$\Gamma$ is
a even cycle of $G$}\rangle.$
\end{thm}
  
A binomial $f = u - v \in I_G$ is called {\it primitive} if there
is no other binomial $g = u'-v'$ in $I_G$ such
that $u'|u$ and $v'|v$.  A closed even walk (or cycle) $\Gamma$  
is a {\it primitive walk (or cycle)} if the corresponding
binomial $f_\Gamma$ is a primitive binomial.  
We have the following refinement of the previous
result:

\begin{thm}[{\cite[Proposition 5.19]{EH}}]\label{thm.universalGB}
If $G$ is a finite simple graph with toric ideal $I_G$, then  
$\{f_\Gamma ~|~ \mbox{$\Gamma$ is a primitive walk}\}$ 
forms a universal Gr\"obner basis for $I_G$, and 
in particular, forms a set of generators of $I_G$.
\end{thm}

For bipartite graphs, computing the
$a$-invariant of $\kk[E]/I_G$ is equivalent
to computing the regularity of $R/I_G$.

\begin{lem}\label{lem.ainvarianttoric}
Let $G = (V,E)$ be a finite simple bipartite graph
with toric ideal $I_G$.  Then
$$
a(\mathbb{K}[E]/I_G) = 
{\rm reg}(\mathbb{K}[E]/I_G) - (|V|-1).
$$
\end{lem}

\begin{proof}
    If $G$ is bipartite, then 
    $\kk[E]/I_G$ is Cohen-Macaulay
    (see \cite[Corollary 5.26]{HHO}).
    By Lemma \ref{lem.reg=deg}, we 
    therefore have ${\rm reg}(\kk[E]/I_G) =
    \deg h_{\kk[E]/I_G}(t)$.
    By 
    \cite[Corollary 10.1.21]{V}
    $\dim(\kk[E]/I_G) = |V|-1$ when 
    $G$ is bipartite.  Thus
    $$a(\kk[E]/I_G) = \deg h_{\kk[E]/I_G}(t)-
    \dim(\kk[E]/I_G) = {\rm reg}(\kk[E]/I_G) -
    (|V|-1),$$
    as desired.
\end{proof}

\subsection{Toric ideals of bipartite graphs}
As our first application, 
we show how Theorem \ref{thm.regformula} can be
used to give a new proof for a result of 
\cite{ADS22} about the regularity of bipartite graphs 
and their subgraphs.  

We first recall some more relevant graph theory. Given a  graph
$G = (V,E)$, we say $H = (W,F)$ is a {\it subgraph}
of $G$ if $W \subseteq V$ and $F \subseteq E$.  
In the special case $H = (V,E\setminus \{e\})$ for
some edge $e$, we write $G\setminus \{e\}$ 
to denote the graph $G$ with the edge $e$ removed.
The {\it degree} of a vertex $x$ is given by
$\deg(x) = |\{y \in V ~|~ \{x,y\} \in E\}|$.  An edge
$e = \{x,y\} \in E$ is a {\it leaf} if $\deg(x)$ or
$\deg(y) = 1$, and a vertex $x$ is {\it isolated}
if $\deg(x) =0$.  If $x$ is an isolated vertex of
$G$, respectively, if $e$ is a leaf of $G$, then it can 
be shown (e.g., see \cite[Lemma 3.2]{CSRV22}) 
that $I_G = I_{G'}$ where 
where $G' = (V \setminus \{x\},E)$, respectively
$G' = G\setminus \{e\}$.

The next lemma applies to all toric ideals of graphs, 
not just bipartite graphs.  

\begin{lem}[{\cite[Lemma 3.5]{CSRV22}}]\label{lem:nyideal}
Let $G$ be a finite simple graph with toric ideal 
$I_G$. If $<$ is any 
$y$-compatible monomial order with y = e for some edge $e \in E$ of $G$, then $N_{y,I_G} = I_{G \setminus \{e\}}.$
\end{lem}

Now suppose that $\mathcal{G}$ is a family of graphs such that for every $G \in \mathcal{G}$, the toric ideal $I_G$ is geometrically vertex decomposable, and for all $e \in E$, there is a
$e$-compatible monomial order so that there is
a  geometric vertex decomposition with respect to $e$
(but each decomposition does not necessarily give geometrically vertex decomposable ideals $C_{e,I_G}$ and $N_{e,I_G}$ as in Definition \ref{defn.gvd}). Furthermore, suppose that $G\setminus \{ e \} \in \mathcal{G}$ for any edge $e$ of $G$.  
For such a family, we have the following result.

\begin{thm}
\label{thm.subgraphreg}
    Let $G$ be any graph in the family $\mathcal{G}$ given
    above. Then for any subgraph $H$ of $G$, we have 
    \begin{enumerate}
        \item $\reg(I_H) \le \reg(I_G)$,
        \item $a(\kk[E(G)]/I_H) \leq a(\kk[E(G)]/I_G)$, and 
        \item $e(\kk[E(G)]/I_H) \leq 
        e(\kk[E(G)]/I_G)$.
    \end{enumerate}
\end{thm}

\begin{proof}
    For any edge $e$, there is an $e$-compatible monomial order such that $I_G$ has the decomposition $\text{in}_e(I_G) = (N_{e,I_G} + \langle e \rangle) \cap C_{e,I_G}$, and by Lemma \ref{lem:nyideal}, we can write the decompostition as
$$
\text{in}_e(I_G) = (I_{G \setminus \{ e \}} + \langle e \rangle) \cap C_{e,I_G}.
$$
As $G\setminus \{ e \} \in \mathcal{G}$, $I_{G \setminus \{ e \}}$ is geometrically vertex decomposable, hence, is Cohen-Macaulay and radical. Because $I_G$ is geometrically vertex decomposable,
by Theorem \ref{thm.regformula} and Remark \ref{rem.IandNCM}, we have $\reg(I_{G \setminus \{ e \}}) \le \reg(I_G)$.  Similarly, Theorem
\ref{thm.aformula} gives
$$a(\kk[E(G)]/I_{G\setminus \{e\}}) < a(\kk[E(G)]/I_{G\setminus \{e\}})+1 \leq
a(\kk[E(G)]/I_G).$$  Moreover,
by Theorem \ref{thm.eformula}
$e(\kk[E(G)]/I_{G\setminus \{e\}}) \leq e(\kk[E(G)]/I_G)$ 
since multiplicity is a non-negative
integer.

Since $G\setminus \{ e \}$ is a graph in $\mathcal{G}$, $I_{G \setminus \{ e \}}$ again has a geometric vertex decomposition with respect to any edge $f$. 
Since any subgraph $H$ of $G$ can be obtained by removing edges and vertices,
by repeating this argument (and possibly removing leaves and isolated vertices when needed)
we get the desired 
conclusion.
\end{proof}

\begin{rem}
    Note that we showed that 
    $a(\kk[E]/I_{G\setminus \{e\}}) \leq a(\kk[E]/I_G)-1$, when we remove the edge
    $e$ from $G$.
    If $t = |E(G)| - |E(H)|$, that is, the number
    of edges we remove from $H$ to form $G$,
    we actually have the stronger result
    $a(\kk[E(G)]/I_H) \leq a(\kk[E(G)/I_G) -t$.
\end{rem}

We can recover \cite[Theorem 6.11]{ADS22} 
which was proved using combinatorial techniques 
involving root polytopes.  We also 
derive results about the $a$-invariant
and multiplicity. 

\begin{thm}
\label{thm.regbipartite}
Let $H$ be any subgraph of a bipartite graph $G$.  Then
\begin{enumerate}
        \item $\reg(I_H) \le \reg(I_G)$,
        \item $a(\kk[E(G)]/I_H) \leq a(\kk[E(G)]/I_G)$, and 
        \item $e(\kk[E(G)]/I_H) \leq 
        e(\kk[E(G)]/I_G)$.
    \end{enumerate}
\end{thm}
\begin{proof}
Let $\mathcal{G}$ be the family of bipartite graphs.
By \cite[Theorem 5.8]{CSRV22}, for all $G \in \mathcal{G}$,
the toric ideal $I_G$ is geometrically vertex decomposable.
It also follows that for any $G \in \mathcal{G}$, 
$G \setminus \{e\} \in \mathcal{G}$, since removing
edges does not destroy the bipartite property.

In addition, by \cite[Proposition 5.4]{CSRV22}, there is a geometric vertex decomposition
$$
\text{in}_e(I_G) = (I_{G \setminus \{ e \}} + \langle e \rangle) \cap I^G_e,
$$
where $\{ e \}$ is a path ordered matching of $G$ and 
$I^G_e = I_{G \setminus \{ e \}} + \langle M^G_e \rangle$, 
and $$M^G_e = \{m \ | \ me-n \text{ is a binomial that corresponds to a cycle 
in $G$}\}.$$ Since any edge of a bipartite graph can be
regarded as a path ordered matching, the above geometric 
vertex decomposition holds for the toric ideal of any 
bipartite graph and any edge.
Consequently, Theorem \ref{thm.subgraphreg}
applies to any subgraph of $H$ of $G$.
\end{proof}

\begin{rem}
Theorem \ref{thm.regbipartite} (1) gives another 
proof of \cite[Theorem 6.11]{ADS22} using the
geometric vertex decomposability property of
toric ideals of bipartite graphs. R. Villarreal
pointed out to the third author that
regularity could also be deduced
from the edge polytope
of $G$ and Stanley's monotonicity 
property \cite[Theorem~3.3]{S}.
In particular, one can use the strategy
given just after \cite[Question 6.12]{ADS22}.
Another proof of Theorem \ref{thm.regbipartite} (1)
can be found in the recent paper
of Vaz Pinto and Villarreal \cite[Corollary 8.17]{VV23}
using normal monomial ideals.  Additionally, there a similar result to Theorem \ref{thm.regbipartite} (1) 
for the regularity of {\it induced}
graphs that can be found in \cite[Theorem 3.6]{HBO2019}.
\end{rem}

%%%%%%%%%%%%%%%%%%%%%%%%%%%%%%%%%%%%%5

\subsection{Regularity and gluing cycles}
We can use Theorem \ref{thm.regformula}
to give a different proof for 
\cite[Corollary 3.11]{FHKV21} which describes 
how regularity
behaves with respect to a ``gluing'' operation
on graphs.

We first recall the notion
of gluing a cycle to a graph along an edge,
following \cite[Construction 4.1]{FHKV21}.
A {\it cycle of length $m$} is the graph 
with vertex set $\{x_1,\ldots,x_m\}$ and
edge set $\{\{x_1,x_2\},\{x_2,x_3\},
\ldots,\{x_{m-1},x_m\},\{x_m,x_1\}\}$; we denote
this graph by $C_m$.   Let $G = (V,E)$
be any graph.  Fix an edge $e \in E$ and
an edge $f$ of $C_m$.  The graph $H$ 
obtained from $G$ by {\it gluing a cycle of length $m$} along
an edge is the graph $G \cup_{e=f} C_m$ where we 
identify the edges and vertices of $e$ and $f$.
An example of gluing is given in Figure \ref{fig:glue}.  When $C_m$ has even length,
the regularity of $I_H$, the toric
ideal of the glued graph $H$, is related
to that of $I_G$, if $I_G$ is geometrically
vertex decomposable.
\begin{figure}[h]
\begin{tikzpicture}[scale=0.6]
 
\draw (0,0) -- (0,3) -- (3,3) -- (3,0) --
(0,0) -- (3,3) -- (0,3) -- (3,0);
\draw (6,0) -- (6,3) -- (9,3) -- (9,0) -- (6,0);
\node at (3.2,1.5) {$e$};
\node at (5.3,1.5) {$f$};

\draw (15,0) -- (15,3) -- (18,3) -- (18,0) --
(15,0) -- (18,3) -- (15,3) -- (18,0);
\draw (18,0) -- (18,3) -- (21,3) -- (21,0) -- (18,0);

\node at (12,1.5) {Gluing $e$ to $f$};
\node at (12,.8) {$\Longrightarrow$};
\node at (19.1,1.5) {$e=f$};

\fill[fill=white,draw=black] (0,0) circle (.2)
node[label=below:$x_1$] {};
\fill[fill=white,draw=black] (3,0) circle (.2)
node[label=below:$x_2$] {};
\fill[fill=white,draw=black] (6,0) circle (.2)
node[label=below:$y_1$] {};
\fill[fill=white,draw=black] (9,0) circle (.2)
node[label=below:$y_2$] {};
\fill[fill=white,draw=black] (0,3) circle (.2)
node[label=above:$x_4$] {};
\fill[fill=white,draw=black] (3,3) circle (.2)
node[label=above:$x_3$] {};
\fill[fill=white,draw=black] (6,3) circle (.2)
node[label=above:$y_4$] {};
\fill[fill=white,draw=black] (9,3) circle (.2)
node[label=above:$y_3$] {};

\fill[fill=white,draw=black] (15,0) circle (.2)
node[label=below:$x_1$] {};
\fill[fill=white,draw=black] (15,3) circle (.2)
node[label=above:$x_4$] {};
\fill[fill=white,draw=black] (18,0) circle (.2)
node[label=below:{$x_2=y_1$}] {};
\fill[fill=white,draw=black] (18,3) circle (.2)
node[label=above:{$x_3=y_4$}] {};
\fill[fill=white,draw=black] (21,0) circle (.2)
node[label=below:$y_2$] {};
\fill[fill=white,draw=black] (21,3) circle (.2)
node[label=above:$y_3$] {};
\end{tikzpicture}
    \caption{Two graphs $G$ and $C_4$ glued
    along edges $e$ and $f$}
    \label{fig:glue}
\end{figure}
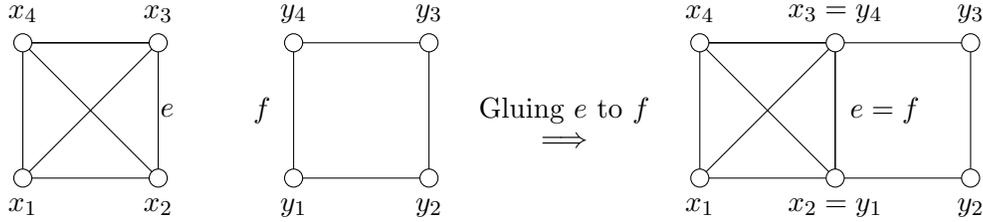

\begin{thm}
\label{thm.gluecycle}
Suppose that $G$ is a graph such that $I_G$ is geometrically vertex decomposable in $\kk[E(G)]$. Let $H$ be the graph obtained from $G$ by gluing a cycle of length $2d$ ($d \geq 2$) along an edge of $G$. Then $\reg(\kk[E(H)]/I_H) =\reg(\kk[E(G)]/I_G)+(d-1)$.
\end{thm}

\begin{proof}
Let $E(G) =\{e_1, \ldots ,e_q\}$ denote the edges of 
$G$ and let $E(C) = \{ f_1, \ldots , f_{2d} \}$ 
denote the edges of the cycle $C = C_{2d}$. We assume
that the cycle is glued to $G$ along $f_{2d}$ and any 
edge of $G$. By \cite[Theorem 3.11]{CSRV22} and its 
proof, $I_H$ is geometrically vertex decomposable and 
moreover, the geometric decomposition is given by 
$N_{y,I_H} = I_G$ and $C_{y,I_H}=I_G + \langle 
f_3f_5\cdots f_{2d-1} \rangle$, where $y=f_1$
and some $y$-compatible monomial order. Since 
$I_G \subset \kk[E(G)]$, $\sqrt{C_{y,I_H}} \neq 
\sqrt{N_{y,I_H}}$ and $C_{y,I_H} \neq \langle 1 
\rangle$, the decomposition is non-degenerate.

Note that $\kk[E(H)] = \kk[E(G)] \otimes T$ where $T=\kk[f_1,\ldots,f_{2d-1}]$.
So
\begin{eqnarray*}
 \reg(\kk[E(H)]/(I_G + \langle f_3f_5\cdots f_{2d-1} \rangle) &=& \reg(\kk[E(G)]/I_G) + \reg(T/\langle f_3f_5\cdots f_{2d-1} \rangle)\\ 
 &= &\reg(\kk[E(G)]/I_G) +(d-2).
 \end{eqnarray*}
So by Theorem \ref{thm.regformula}, we have
\begin{eqnarray*}
\reg(\kk[E(H)]/I_H) &=& 
\max \{\reg(\kk[E(G)]/I_G), \reg(\kk[E(G)]/I_G)+(d-2)+1 \} \\
&=& \reg(\kk[E(G)]/I_G)+(d-1),
\end{eqnarray*}
thus completing the proof.
\end{proof}

We can now derive the following corollary
for gluing even cycles to
graphs $G$ such that $I_G$ is geometrically
vertex decomposable.
%\begin{cor}
%    Suppose that $G$ is a bipartite grap and let $H$ be the graph obtained from $G$ by gluing a cycle of length $2d$ along an edge of $G$. Then $a(\kk[E(H)]/I_H)  =a(\kk[E(G)]/I_G)+(d-1)$.
%\end{cor}

%\begin{proof}
%    Note that $|V(H)| = |V(G)|+(2d-2)$.
%    Since $G$ is bipartite, $I_G$ is
%    geometrically vertex decomposable.  Thus,
%    by Lemma \ref{lem.ainvarianttoric} and 
%    Theorem \ref{thm.gluecycle} we have
%    \begin{eqnarray*}
%        a(\kk[E(H)]/I_H) & = & 
%        {\rm reg}(\kk[E(H)]/I_H) - (|V(H)|-1) %\\
%        & = & ({\rm reg}(\kk[E(G)]/I_G) + %(d-1))
%        - ((|V(G)|+(2d-2))-1) \\
%        & = & a(\kk[E(G)]/I_G) -(d-1),
%            \end{eqnarray*}
%            as desired.
%\end{proof}

%\textcolor{blue}{I think the same proof holds for geometrically vertex decomposable graphs.}

\begin{cor}
    \label{cor.aegluecycle}
    Suppose that $G$ is a graph such that $I_G$ is geometrically vertex decomposable in $\kk[E(G)]$. Let $H$ be the graph obtained from $G$ by gluing a cycle of length $2d$ ($d\ge 2$) along an edge of $G$. Then
    \begin{enumerate}
        \item $a(\kk[E(H)]/I_H)  =a(\kk[E(G)]/I_G)-(d-1)$.
        \item $e(\kk[E(H)]/I_H)  =d\cdot e(\kk[E(G)]/I_G)$.
    \end{enumerate}    
\end{cor}
\begin{proof}
    Note that $|V(H)| = |V(G)|+(2d-1)$.
    By Lemma \ref{lem.ainvarianttoric} and 
    Theorem \ref{thm.gluecycle} we have
    \begin{eqnarray*}
        a(\kk[E(H)]/I_H) & = & 
        {\rm reg}(\kk[E(H)]/I_H) - (|V(H)|-1) \\
        & = & ({\rm reg}(\kk[E(G)]/I_G) + (d-1))
        - ((|V(G)|+(2d-2))-1) \\
        & = & a(\kk[E(G)]/I_G) -(d-1),
            \end{eqnarray*}
            hence, part $(1)$ follows.
            
        For part $(2)$, by Theorem \ref{thm.eformula}, 
        $$e(\kk[E(H)]/I_H) = e(\kk[E(H)]/I_G) + e(\kk[E(H)]/(I_G+\langle f_3f_5\cdots f_{2d-1} \rangle)).$$ Since $e(\kk[E(H)]/I_G) = e(\kk[E(G)]/I_G)$ and  
    \begin{align*}
       e(\kk[E(H)]/(I_G+\langle f_3f_5\cdots f_{2d-1} \rangle)) &= e(\kk[E(G)]/I_G)\cdot e(\kk[f_1,f_2,\ldots ,f_{2d-1}]/\langle f_3f_5\cdots f_{2d-1} \rangle)\\
       &= (d-1)e(\kk[E(G)]/I_G),
    \end{align*}
    it follows that $e(\kk[E(H)]/I_H) = d\cdot e(\kk[E(G)]/I_G)$ as desired. 
    \end{proof}

%%%%%%%%%%%%%%%%%%%%%%%%%%%%%%%%%%%%%%%

\subsection{Regularity of toric ideals of 
Ferrers graphs}

Recall that a {\it Ferrers graph} is a bipartite graph on the 
vertex set $X=\{ v_1,v_2,\ldots ,v_n \}$ and 
$Y=\{u_1, u_2,\ldots ,u_m \}$ such that 
$\{ v_1, u_m\}$ and $\{v_n, u_1\}$ are edges and 
if $\{v_i,u_j \}$ is an edge, then so are all the edges 
$\{v_k,u_l \}$ with $1\le k \le i$ and $1\le l \le j$. 
We also associate a partition 
$\lambda = (\lambda_1, \lambda_2, \ldots ,\lambda_n)$ with 
$\lambda_1\ge \lambda_2 \ge \ldots \ge \lambda_n$ to a Ferrers graph where $\lambda_i = \deg v_i$, and we
denote the Ferrers graph $T_\lambda$.  See 
Figure \ref{fig_ferrers} for an example.
We show how to use Theorem \ref{thm.regformula} to give a
different proof to a  result
of Corso and Nagel \cite{CN09}.

We first require a lemma; in the statement
below, a bipartite graph is a {\it chordal bipartite} graph
if every cycle of length $\geq 6$ has a chord, that is,
an edge that joins two non-consecutive vertices of the cycle.

\begin{lem}\label{lem.ferrerstoric}\label{lem.quadratics}
Suppose that $T_\lambda$ is a Ferrers graph.  Then
$T_\lambda$ is a chordal bipartite graph. Consequently,
the toric ideal $I_{T_\lambda}$ is generated
by quadratics. 
\end{lem}

\begin{proof}
   Let $X = \{v_1,\ldots,v_n\}$ and $Y=\{u_1,\ldots,u_m\}$ be
   the partition of the vertices of $T_\lambda$.
   Suppose that $(v_{i_1},u_{j_1},v_{i_2},u_{j_2},\ldots,
   v_{i_s},u_{j_s})$ are the vertices of a cycle of length
   $2s \geq 6$ in $T_\lambda$.

   The three edges $\{v_{i_1},u_{j_1}\},\{v_{i_1},u_{j_s}\},\{v_{i_s},u_{j_s}\}$ appear in this cycle. Consider the 
   indices of the two $u$ vertices.  If $j_1 < j_s$, then
   by the definition of $T_\lambda$, the
   edge $\{v_{i_s},u_{j_1}\}$ is also an edge of $T_\lambda$.
   So the cycle has a chord.  If $j_s < j_1$, note that
   $\{v_{i_2},u_{j_1}\}$ is the next edge in the cycle.
   Since $j_s < j_1$, the edge $\{v_{i_2},u_{j_s}\}$
   is also an edge of $T_\lambda$.  But then
   $(v_{i_1},u_{j_s},v_{i_2},u_{j_1})$ is a four cycle of
   $T_\lambda$, that is, $\{v_{i_2},u_{j_s}\}$ is a chord.

    The final statement follows from the main 
    result of \cite{OH} which showed that the 
    toric ideals of all chordal bipartite graphs 
    are generated by quadratics.
\end{proof}

\begin{thm}[{\cite[Proposition 5.7]{CN09}}]
\label{thm.Ferrer}
Let $\lambda = (\lambda_1, \lambda_2, \ldots ,\lambda_n)$ 
be a partition with $\lambda_1\ge \lambda_2 \ge \cdots \ge 
\lambda_n$ and let $T_\lambda$ be the associated Ferrers graph. 
Denote $I_\lambda = I_{T_\lambda}$ and $R = \kk[E(T_\lambda)].$
\begin{enumerate}
    \item If $n=1$ or $\lambda_2=1$, then $\reg(R/I_\lambda)=0$.
    \item If $\lambda_2 \ge 2$, and suppose that $\lambda = 
    (\lambda_1, \ldots , \lambda_s ,1,1,\ldots ,1)$ where 
    $\lambda_s \ge 2$, then 
    $$\reg(R/I_\lambda) = \min \{ s-1, \{\lambda_j+j-3 \ | \ 2\le 
    j \le s \} \}.$$
\end{enumerate}
\end{thm}

\begin{proof}
If $n=1$, $I_\lambda= \langle 0 \rangle$, hence 
$\reg(R/I_\lambda)=0$. 
When $\lambda_2=1$, all the edges $\{v_i,u_1 \}$ with $i\ge 2$ 
are leaves, hence, after removing them, we have $I_\lambda = 
 \langle 0 \rangle$.

Now suppose that $\lambda_2 \ge 2$. Denote 
$e_{ij} = \{v_i , u_j \}$ and let 
$e = e_{n\lambda_n} =\{v_n , u_{\lambda_n} \}$. 
Let $G=T_\lambda$, and $I_G=I_\lambda$. 
By \cite[Theorem 5.8]{CSRV22}, $I_G$ is 
geometrically vertex decomposable since $G$ is bipartite. Moreover, by \cite[Proposition 5.4]{CSRV22}, there is a 
geometric vertex decomposition
\[
\text{in}_e(I_G) = (I_{G \setminus \{ e \}} + \langle e \rangle) \cap I^G_e,
\]
where $I^G_e = I_{G \setminus \{ e \}} +  \langle M^G_e\rangle$, 
and $M^G_e = \{m_1 \ | \ m_1e-m_2 \text{ corresponds to a cycle in 
$G$}\}$ (the description of $M_e^G$ appears directly
after \cite[Lemma 5.1]{CSRV22}).

We claim that 
$\langle M^G_e \rangle = \langle \{e_i \ | \ \{e,e_j,e_i,e_k\}  
\text{ is a cycle in $G$} \}\rangle$. 
Note that it suffices
to verify that $\langle M^G_e \rangle \subseteq 
 \langle \{e_i \ | \ \{e,e_j,e_i,e_k\} 
\text{ is a cycle in $G$} \}\rangle$ since the reverse
containment is immediate.  Suppose that the
binomial $m_1e-m_2$ corresponds to a cycle of $G$.  Since
the cycle contains the edge $e = \{v_n,u_{\lambda_n}\}$, we can write this
cycle as $$(v_{i_1},u_{j_1},v_{i_2},u_{j_2},\ldots,
v_{i_{s-1}},u_{i_{s-1}},v_{n},u_{\lambda_n})$$
where $i_s =n$ and and $j_s = \lambda_n$. Furthermore,
we denote the edges in the cycle as follows:
$e_k = \{v_{i_k},u_{j_k}\}$ for $k = 1,\ldots,s$ and
$f_k = \{u_{j_k},v_{i_{k+1}}\}$ for $k=1,\ldots,s$ where 
$v_{i_{s+1}}=v_{i_1}$. Note that $e=e_{i_s}$.  
With this notation, we have
$m_1e-m_2 = e_1e_2\cdots e_{s-1}e - f_1f_2\cdots f_s$.
In this cycle, consider the three consecutive edges
$e_{{s-1}} =\{v_{i_{s-1}},u_{i_{s-1}}\},f_{{s-1}}
= \{u_{i_{s-1}},v_n\}$ and $e = \{v_n,u_{\lambda_n}\}$.
Since $T_\lambda$ is a Ferrers graph, and since
$i_{s-1} < n$, the edge $f = \{v_{i_{s-1}},u_{\lambda_n}\}$
also belongs to $T_\lambda$.  This gives a four cycle
$(e_{{s-1}},f_{{s-1}},e,f)$, and thus
$e_{s-1}e-f_{s-1}f$ is a binomial that corresponds to
a cycle of $G$. But
this means that $e_{{s-1}} \in M_e^G$, and since
$e_{{s-1}}$ divides $m_1 = e_1\cdots e_{s-1}$, we 
have $m_1$ is in the ideal on the right-hand side.

Note that the graph $G\setminus \{ e \}$ is the Ferrers graph on  
the same vertex set associated to the partition $\lambda''=
(\lambda_1,\lambda_2,\ldots, \lambda_{n-1},\lambda_n-1)$.  Hence, 
$N_{e,I_G}=I_{G\setminus \{e\}}=I_{\lambda''}.$
On the other hand, 
$$\langle M^G_e \rangle= \langle e_{ij} \ | \ 1\le i \le n-1, 1\le j \le \lambda_n-1 \rangle.$$
Since $G\setminus\{e\}$ is a Ferrers graph, 
the generators of $I_{G\setminus \{e\}}$ are
quadratics by Lemma \ref{lem.quadratics}. Moreover
these generators are of the form $f=e_{i_1j_1}e_{i_2j_2}-e_{i_2j_1}e_{i_1j_2}$ with $i_1<i_2$ and $j_1<j_2$. If $i_2\le n-1$ and $j_2\le \lambda_n-1$, then $f\in \langle M^G_e \rangle$. If $i_2=n$, then $j_1<j_2\le \lambda_n-1$, hence $f\in \langle M^G_e \rangle$. Therefore,
$$C_{e,I_G}=I_{G \setminus \{ e \}} + \langle M^G_e \rangle =\langle M^G_e \rangle+ \langle e_{i_1j_1}e_{i_2j_2}-e_{i_2j_1}e_{i_1j_2} \ | \ 1\le i_1<i_2 \le n-1, \lambda_n \le j_1 <j_2 \rangle.$$
Note that the ideal 
$\langle e_{i_1j_1}e_{i_2j_2}-e_{i_2j_1}e_{i_1j_2} \ | \ 1\le i_1<i_2 \le n-1, \lambda_n \le j_1 <j_2 \rangle$ 
is the toric ideal of the Ferrers graph on the vertex set 
$X'=\{ v_1,v_2,\ldots ,v_{n-1} \}$ and 
$Y'=\{u_{\lambda_n}, u_2,\ldots ,u_m \}$ associated to 
the partition $\lambda'=(\lambda_1-\lambda_n+1,\lambda_2-\lambda_n+1,\ldots , \lambda_{n-1}-\lambda_n+1)$. 
Thus, we can write $C_{e,I_G}= \langle M^G_e \rangle + I_{\lambda'}$, where the generators of the two ideals in the right-hand side are in separate sets of variables. 
Moreover, since $\langle M_e^G \rangle$ is generated
by variables, ${\rm reg}(R/ \langle M^G_e \rangle + I_{\lambda'}) = {\rm reg}(R/I_\lambda')$. Thus, by Theorem \ref{thm.regformula},
\[
\reg(R/I_\lambda) = \max \{\reg(R/I_{\lambda''}) , \reg(R/I_{\lambda'})+1 \},
\]
where $\lambda'$ and $\lambda''$ are defined as above. 

We will now apply the above recursive formula to derive the 
formula of $\reg(R/I_\lambda)$ as claimed. First, if $\lambda = 
(\lambda_1, \ldots , \lambda_s ,1,1,\ldots ,1)$ with $\lambda_s 
\ge 2$, then the edges $v_iu_1$ with $s+1\le i \le n$ are all 
leaves, hence, we can remove them without changing the toric 
ideal. In other words, $\reg(I_\lambda) 
=\reg(I_{\tilde{\lambda}})$ where $\tilde{\lambda} = 
(\lambda_1, \ldots , \lambda_s)$. 
We will use induction on $n$. Suppose $n=2$ and $\lambda_1\ge \lambda_2 \ge 2$. Then, by the recursive formula,
\begin{align*}
\reg(R/I_{(\lambda_1,\lambda_2)}) &= \max \{\reg(R/I_{(\lambda_1,\lambda_2-1)}) , \reg(R/I_{(\lambda_1-\lambda_2+1)})+1 \} \\
&= \max \{\reg(R/I_{(\lambda_1,\lambda_2-1)}) , 1 \}.
\end{align*}
By induction on $\lambda_2$, if $\lambda_2 =2$, we have
\[
\reg(R/I_{(\lambda_1,\lambda_2)})= \max \{\reg(R/I_{(\lambda_1,1)}) , 1 \}=1 =\min \{ 2-1, \lambda_2+2-3 \},
\]
and if $\lambda_2 > 2$, we have
\begin{eqnarray*}
\reg(R/I_{(\lambda_1,\lambda_2)}) & =& \max \{\reg(R/I_{(\lambda_1,\lambda_2-1)}) , 1 \} \\
&=&
\max\{\min\{2-1,(\lambda_2-1)+2-3\}\},1\} = 1 =\min \{ 2-1, \lambda_2+2-3 \}.
\end{eqnarray*}
In both case, the regularity agrees with the formula in our claim. 

Now suppose that $n\ge 3$, and the formula holds for any 
$k \le n-1$ and for any $\lambda=(\lambda_1,\lambda_2,\ldots ,\lambda_k)$. 
We will show that it holds for $n$ and any 
$\lambda=(\lambda_1,\lambda_2,\ldots ,\lambda_n)$. 
Now if $\lambda = (\lambda_1, \ldots , \lambda_s ,1,1,\ldots ,1)$ 
with $\lambda_s\ge 2$, as shown above, we 
have $\reg(R/I_\lambda)=\reg(R/I_{\tilde{\lambda}})$, 
where $\tilde{\lambda} = (\lambda_1, \ldots , \lambda_s)$. 
Hence, if $s\le n-1$, the result follows by the induction hypothesis. 
Also, if $\lambda_n=1$, as $\reg(R/I_\lambda)= \reg(R/I_{(\lambda_1, \ldots , \lambda_{n-1})})$, 
the result again follows by the induction hypothesis. 
Thus, we can assume that $s=n$ and $\lambda_n\ge 2$. Then, by the recursive formula,
\begin{align*}
    \reg(R/I_\lambda)&= \max \{\reg(R/I_{(\lambda_1, \ldots , \lambda_{n-1}, \lambda_n-1)}), \reg(R/I_{(\lambda_1-\lambda_n+1, \ldots , \lambda_{n-1}-\lambda_n+1)}) +1 \}.
\end{align*}
\textbf{Case 1:} If $\lambda_n=2$, then $\reg(R/I_{(\lambda_1, \ldots , \lambda_{n-1}, \lambda_n-1)}) = \reg(R/I_{(\lambda_1, \ldots , \lambda_{n-1})})$. So, by the induction hypothesis
\[
\reg(R/I_{(\lambda_1, \ldots , \lambda_{n-1}, \lambda_n-1)}) = \min \{ n-2, \{\lambda_j+j-3 \ | \ 2\le j \le n-1 \} \}.
\]
Now if $\lambda_{n-1}\ge 3$, then by the induction hypothesis
\begin{align*}
    \reg(R/I_{(\lambda_1-\lambda_n+1, \ldots , \lambda_{n-1}-\lambda_n+1)})+1
    &= \min \{ n-2, \{\lambda_j-\lambda_n+j-2 \ | \ 2\le j \le n-1 \} \} +1 \\
    &= \min \{ n-1, \{\lambda_j+j-3 \ | \ 2\le j \le n-1 \} \}\\
    &\ge \reg(R/I_{(\lambda_1, \ldots , \lambda_{n-1}, \lambda_n-1)}).
\end{align*}
Therefore,
\[
\reg(R/I_\lambda)=\reg(R/I_{(\lambda_1-\lambda_n+1, \ldots , \lambda_{n-1}-\lambda_n+1)})+1 =  \min \{ n-1, \{\lambda_j+j-3 \ | \ 2\le j \le n \} \},
\]
where the last equality holds since $\lambda_n+n-3=n-1$.\par
%\vspace{0.5em}

Otherwise, if $\lambda_{n-1}=2$, assume that $\lambda_{k+1}=\lambda_{k+2}=\ldots =\lambda_{n-1}=2$, and $\lambda_k\ge 3$ for some $k\le n-2$.  
Then $(\lambda_1-\lambda_n+1,\ldots,\lambda_{n-1}-\lambda_n+1) = (\lambda_1-1,\ldots,\lambda_k-1,1,\ldots,1)$.
Again, by the induction hypothesis,
\begin{align*}
    \reg(R/I_{(\lambda_1-\lambda_n+1, \ldots , \lambda_{n-1}-\lambda_n+1)})+1 &=
    \min \{k-1,\{\lambda_j-1+j-3 \ | \ 2 \leq j \leq k\}\} + 1 \\
&=    \min \{k,  \{\lambda_j+j-3 \ | \ 2\le j \le k \}\},
\end{align*}
if $k\ge 2$ or the regularity equals $0$ if $k= 1$. Moreover, for $k+1\le j\le n-1$, $k\le \lambda_j+j-3\le n-2$, thus,
\begin{align*}
    \min \{k,  \{\lambda_j+j-3 \ | \ 2\le j \le k \}\} &= \min \{k,  \{\lambda_j+j-3 \ | \ 2\le j \le n-1 \}\}\\
    &\le \min \{ n-2, \{\lambda_j+j-3 \ | \ 2\le j \le n-1 \} \}\\
    &=\reg(R/I_{(\lambda_1, \ldots , \lambda_{n-1}, \lambda_n-1)}).
\end{align*}
where the last equality follows since $\lambda_j+j-3\le n-2<n-1$ for $k+1\le j\le n-1$. Therefore, 
\[
\reg(R/I_\lambda)=\reg(R/I_{(\lambda_1, \ldots , \lambda_{n-1}, \lambda_n-1)})=\min \{ n-1, \{\lambda_j+j-3 \ | \ 2\le j \le n \} \},
\]
where the last equality follows since $\lambda_n+n-3=n-1$.\par
\vspace{0.5em}

\textbf{Case 2:} If $\lambda_n\ge 3$, by induction on $\lambda_n$, we have
\[
\reg(R/I_{(\lambda_1, \ldots , \lambda_{n-1}, \lambda_n-1)}) = \min \{ n-1, \{\lambda_j+j-3 \ | \ 2\le j \le n-1 \}, \lambda_n+n-4 \}.
\]
In addition, by the induction hypothesis (on $n$), we have 
\[
\reg(R/I_{(\lambda_1-\lambda_n+1, \ldots , \lambda_{n-1}-\lambda_n+1)})= \min \{ k-1, \{\lambda_j-\lambda_n+j-2 \ | \ 2\le j \le k \} \},
\]
where $2\le k \le n-1$ is the maximum integer such that $\lambda_k\ge \lambda_n +1$ (hence, $\lambda_k - \lambda_n +1\ge 2$), or the regularity equals $0$ if $k=1$. On the other hand, as $\lambda_j =\lambda_n$ for $j\ge k+1$, we have $\lambda_j-\lambda_n+j-2=j-2\ge k-1$, for all $j\ge k+1$. Thus,
\[
\reg(R/I_{(\lambda_1-\lambda_n+1, \ldots , \lambda_{n-1}-\lambda_n+1)})= \min \{ k-1, \{\lambda_j-\lambda_n+j-2 \ | \ 2\le j \le n \} \}.
\]
Therefore,
\begin{align*}
    \reg(R/I_{(\lambda_1-\lambda_n+1, \ldots , \lambda_{n-1}-\lambda_n+1)})+1
    &= \min \{ k, \{\lambda_j-\lambda_n+j-1 \ | \ 2\le j \le n \} \}\\
    &\le \min \{ n-1, \{\lambda_j+j-3 \ | \ 2\le j \le n \} \}\\
    &=\reg(R/I_{(\lambda_1, \ldots , \lambda_{n-1}, \lambda_n-1)}),
\end{align*}
where the last equality holds since $\lambda_n+n-3>\lambda_n+n-4\ge n-1$. Therefore,
\[
\reg(R/I_\lambda)=\reg(R/I_{(\lambda_1, \ldots , \lambda_{n-1}, \lambda_n-1)})=\min \{ n-1, \{\lambda_j+j-3 \ | \ 2\le j \le n \} \}.
\]
finishing our proof.
\end{proof}

\begin{rem}
\label{rem.hilbertseries}
 If $y$ is the variable that
    corresponds to the edge $e = \{v_n,\lambda_n\}$, by Theorem
 \ref{thm.hilbertseries} we have
    \[
    H_{R/I_\lambda}(t) = H_{R/N_{y,I}+\langle y \rangle}(t) + tH_{R/C_{y,I}}(t)=H_{R/I_\lambda''}(t)+tH_{R/I_\lambda'}(t)
    \]
    This recovers \cite[Lemma 5.3]{CN09}.
\end{rem}

\begin{rem}
    \label{rem.Ferrerreg}
    As shown in the proof of Theorem \ref{thm.Ferrer}, we record the formulae for regularity of $I_\lambda$ avoiding taking 
    the maximum:
    \begin{enumerate}
        \item If $\lambda_n=2$ and $\lambda_{n-1}\ge 3$, then  
        $\reg(I_\lambda)=\reg(C_{e,I_\lambda})+1=\reg(I_\lambda') +1$.
        \item If $\lambda_{n-1}=\lambda_n=2$, then 
        $\reg(I_\lambda)=\reg(N_{e,I_\lambda})=\reg(I_\lambda'')$.
        \item If $\lambda_n\ge 3$, then 
        $\reg(I_\lambda)=\reg(N_{e,I_\lambda})=\reg(I_\lambda'')$.
    \end{enumerate}
\end{rem}

As mentioned in Remark \ref{rem.regeitherCN}, the following example shows that the $a$-invariant of a geometrically vertex decomposable ideal $I$ can be either that of the $C$ ideal or the $N$ ideal of its decomposition, or equivalently,
the regularity of $I$ can be either $\reg(C)+1$ or $\reg(N)$).

\begin{ex}
\label{ex.CNFerrers}
\begin{itemize}
    \item Consider the toric ideal $I = I_\lambda$ of the Ferrers graph associated to $\lambda=(3,3,3,3)$. Then $I$ has a nondegenerate geometric vertex decomposition with $N=I_{\lambda''}$ and $C=I_{\lambda'}$, where $\lambda''=(3,3,3,2)$ and $\lambda'=(1,1,1)$, see the proof of Theorem \ref{thm.Ferrer}. Thus, by Theorem \ref{thm.Ferrer}, $\reg(N)=3,$ $\reg(C)=1$, and $\reg(I)=\reg(N)=3 > \reg(C)+1$. Note also that it is not hard to construct an example with a degenerate decomposition with $C = \langle 1 \rangle$, in this case, we also have $\reg(I)=\reg(N) > \reg(C)+1$.
    \item If $\lambda=(4,4,3,2)$, we have $\lambda''=(4,4,3,1)$ and $\lambda'=(3,3,2)$. By Theorem \ref{thm.Ferrer}, $\reg(N)=2, ~\reg(C)=2$, and $\reg(I)=\reg(C)+1=3 > \reg(N)$.
    \item If $\lambda=(3,3,2,2)$, we have $\lambda''=(3,3,2,1)$ and $\lambda'=(2,2,1)$. By Theorem \ref{thm.Ferrer}, $\reg(N)=2, ~\reg(C)=1$, and $\reg(I)=\reg(C)+1=\reg(N)=2$.
\end{itemize}
\end{ex}

\begin{rem}
    Two formulae for the regularity
    of $R/I_\lambda$ are given in \cite[Proposition 5.7]{CN09}; however these two
    formulas to do not agree in general, that
    is, $$\min \{ s-1, \{\lambda_j+j-3 \ | \ 2\le j \le s \} \} \neq \begin{cases}
    s-1 \quad \quad \quad \quad \quad \quad \ \ \text{if  } \lambda_s\ge 3 \\
    \min\{j-1 \ | \ \lambda_j=2 \} \text{  if  } \lambda_s = 2.
     \end{cases}$$
     For example, let $s\ge 4$, $\lambda_1=\cdots =\lambda_{s-1}=3$ and $\lambda_s=2$. Then the left-hand side is $2$, whereas the right-hand side is $s-1\ge 3$.
     For example, consider the Ferrers graph $T_\lambda$
     for $\lambda = (3,3,3,2)$ as given in
     Figure \ref{fig_ferrers}.
     \begin{figure}[!ht]
    \centering
    \begin{tikzpicture}[scale=0.45]
      % graph G
      \draw (0,0) -- (0,5);
      \draw (5,0) -- (0,5);
       \draw (10,0) -- (0,5);
      \draw (15,0) -- (0,5);
       \draw (15,0) -- (5,5);
       %\draw (20,0) -- (0,5);
       \draw (0,0) -- (5,5);
      \draw (5,0) -- (5,5);
       \draw (10,0) -- (5,5);
      \draw (0,0) -- (10,5);
       \draw (5,0) -- (10,5);
       %\draw (0,0) -- (15,5);
       \draw (10,0) -- (10,5);
       
      \fill[fill=white,draw=black] (0,0) circle (.1) node[below]{$v_1$};
       \fill[fill=white,draw=black] (5,0) circle (.1) node[below]{$v_2$};
       \fill[fill=white,draw=black] (10,0) circle (.1) node[below]{$v_3$};
       \fill[fill=white,draw=black] (15,0) circle (.1) node[below]{$v_4$};
       \fill[fill=white,draw=black] (0,5) circle (.1) node[above]{$u_1$};
       \fill[fill=white,draw=black] (5,5) circle (.1) node[above]{$u_2$};
        \fill[fill=white,draw=black] (10,5) circle (.1) node[above]{$u_3$};
       \end{tikzpicture}
    \caption{The graph $T_\lambda$ for $\lambda = (3,3,3,2)$}
    \label{fig_ferrers}
  \end{figure}
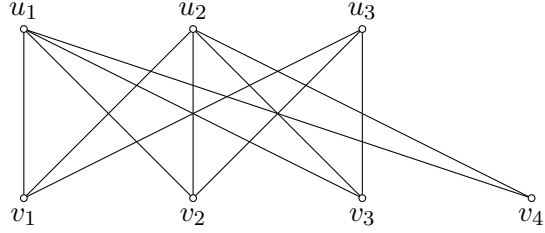
     
     \noindent
     Then the graded minimal free resolution
     of $R/I_\lambda$ has Betti table 
$$\begin{matrix}
   & 0 & 1 & 2 & 3 & 4 & 5 \\
\text{total:}
   & 1 & 12 & 25 & 21 & 10 & 3 \\
0: & 1 & . & . & . & . & . \\
1: & . & 12 & 25 & 15 & . & .\\
2: & . & . & . & 6 & 10 & 3 
\end{matrix}$$
In particular, ${\rm reg}(R/I_\lambda) =2.$
\end{rem}
\begin{cor}
    \label{cor.aeFerrer}
    With the same notation as in Theorem \ref{thm.Ferrer}, we have
    \begin{enumerate}
        \item $a(R/I_\lambda) = \begin{cases}
            -(n+m-1) \text{  if } n=1 \text{ or } \lambda_2=1 \\
            -(n+m-1) + \min \{ s-1, \{\lambda_j+j-3 \ | \ 2\le 
    j \le s \} \} \text{ otherwise}.
        \end{cases}$
        \item $e(R/I_\lambda)= \displaystyle \sum_{j_{n-2}=\lambda_2-\lambda_n+1}^{\lambda_2} \sum_{j_{n-3}=\lambda_2-\lambda_{n-1}+1}^{j_{n-2}} \cdots \sum_{j_{1}=\lambda_2-\lambda_{3}+1}^{j_{2}} j_1.$
    \end{enumerate}
\end{cor}
\begin{proof}
    Part $(1)$ follows from Lemma \ref{lem.ainvarianttoric} and Theorem \ref{thm.Ferrer}, while part $(2)$ follows by induction using Theorem \ref{thm.eformula} as
    \begin{align*}
        e(R/I_\lambda)&=e(R/I_{(\lambda_1, \ldots , \lambda_{n-1}, \lambda_n-1)}) + e(R/I_{(\lambda_1-\lambda_n+1, \ldots , \lambda_{n-1}-\lambda_n+1)})\\
        &= \displaystyle \sum_{j_{n-2}=\lambda_2-\lambda_n}^{\lambda_2} \sum_{j_{n-3}=\lambda_2-\lambda_{n-1}+1}^{j_{n-2}} \cdots \sum_{j_{1}=\lambda_2-\lambda_{3}+1}^{j_{2}} j_1 
        + 
        \sum_{j_{n-3}=\lambda_2-\lambda_{n-1}+1}^{\lambda_2-\lambda_n+1} \cdots \sum_{j_{1}=\lambda_2-\lambda_{3}+1}^{j_{2}} j_1\\
        &= \displaystyle \sum_{j_{n-2}=\lambda_2-\lambda_n+1}^{\lambda_2} \sum_{j_{n-3}=\lambda_2-\lambda_{n-1}+1}^{j_{n-2}} \cdots \sum_{j_{1}=\lambda_2-\lambda_{3}+1}^{j_{2}} j_1.
    \end{align*}
\end{proof}

\begin{rem}
\label{rem.completebipartite}
Let $G=K_{n,m}$ be the complete bipartite graph. Then $I_G=I_\lambda$ where $\lambda=(m,m,m,\ldots ,m)$. Hence, $\reg(R/I_G) = \min \{ n-1,m-1 \} = \min \{ n,m \}-1.$
\end{rem}

We recover the bounds for the regularity of bipartite graphs in \cite[Theorem 6.13]{ADS22} for connected bipartite graphs and \cite[Theorem 4.9]{BOVT17} for chordal bipartite graphs, and the $a$-invariant  of bipartite graphs
in \cite[Proposition 11.5.1]{V}. 

\begin{cor}
\label{cor.regbipartite}
Let $G$ be a connected bipartite graph with
bipartition $V_1 \cup V_2$ with $|V_1|=n$ and $|V_2|=m$.
Let $r=| \{v_i\in V_1 ~|~ \deg(v_i)=1\}| $ and $s=| \{u_i\in V_2 ~|~ \deg(u_i)=1\}| $. Then 
\begin{enumerate}
\item $\reg(I_G)\le \min \{n-r,m-s\}$, and 
\item $a(\kk[E(G)]/I_G) \leq \min\{-m,-n\}$
\end{enumerate}
\end{cor}
\begin{proof}
    $(1)$
    Since $r$ vertices in vertex set $X$ and $s$ vertices in vertex set $Y$ belong to leaves, we can remove these vertices without changing the toric ideal $I_G$. Thus, we can assume that $r=s=0$, that is $G$ does not have any leaves. As $G$ is a subgraph of $K_{n,m}$, by Theorem \ref{thm.regbipartite} and Remark \ref{rem.completebipartite}, we have $\reg(I_G)\le \reg(I_{K_{n,m}}) = \min \{ n,m \}$ as desired.

    $(2)$ By Theorem \ref{lem.ainvarianttoric}
    $a(\kk[E(G)]/I_G = {\rm reg}(\kk[E(G)]/I_G) - (|V(G)| -1)$.
    Because the graph is connected, $|V(G)|=m+n$.  Also, by
    part $(1)$, ${\rm reg}(\kk[E(G)]/I_G) \leq \min\{n,m\}+1$.  The
    result now follows.
    \end{proof}

We can now prove a 
very interesting property about toric ideals of
bipartite graphs, which does not seem to have been
observed before.

\begin{thm}\label{thm.toricbipartitehilbertan}
If $G$ is a connected bipartite graph, then $I_G$ is Hilbertian.
\end{thm}

\begin{proof}
    As shown in Corollary \ref{cor.regbipartite} (2),
    the $a$-invariant of $\kk[E(G)]/I_G$ is always 
    negative, hence $I_G$ is Hilbertian.
\end{proof}
%%%%%%%%%%%%%%%%%%%%%%%%%%%%%%%%%%%%%%%%%%%%%

\subsection{Invariants for a class of 
bipartite graphs}
In this section we apply 
Theorems \ref{thm.regformula}, \ref{thm.aformula},
and \ref{thm.eformula} to a family 
the toric ideals of bipartite
graphs first studied in
\cite{GHKKPVT} to further illustrate
our techniques. The graphs studied in 
\cite{GHKKPVT} were
defined as follows.

\begin{defn}\label{defn:specialgraph}
Let $d,r$ be integers such that 
$d \geq 1$ and $r\geq 3$.
Let $G=G_{r,d}$ be the graph with the vertex
set
$V(G) = \{x_1,x_2,y_1,\ldots,y_d,
z_1,\ldots,z_{2r-3}\}$
and edge set 
\begin{eqnarray*}
    E(G) &= &\{\{x_i,y_j\} ~|~ 1 \leq i \leq 2,~ 1 \leq j \leq
    d \}  
    \\
    && \cup~ \{\{x_1,z_1\},\{z_{1},z_{2}\},\{z_{2},z_{3}\},
    \ldots, \{z_{2r-4},z_{2r-3}\},\{z_{2r-3},x_2\}\}. 
\end{eqnarray*}
We label our edges as 
follows: for $i =1,\ldots,d$, let 
$a_i = \{x_1,y_i\}$ and $b_i = \{x_2,y_i\}$.
Also, let $e_{1} = \{x_1,z_{1}\}$, 
$e_{2r-2} = \{x_2,z_{2r-3}\}$, and 
let $e_{2} = \{z_{1},z_{2}\}, 
e_{3} = \{z_{2},z_{3}\}, \ldots, e_{2r-3}
= \{z_{2r-4},z_{2r-3}\}$.
\end{defn}

Informally, the graph $G_{r,d}$ is 
constructed by
starting with a complete bipartite graph $K_{2,d}$ (with
$d \geq 1$).  Then one connects the vertex 
$x_1$ to $x_2$ (the two vertices of degree $d$)
with a path of length $2r-2$. 
Figure \ref{fig:example_graph} shows the graph
$G_{6,5}$ where we start with the graph
$K_{2,5}$ (the graph on $\{x_1,x_2,y_1,\ldots,y_5\})$, and
then add a path of length $2\cdot 6-2$ between $x_1$ and $x_2$,
the two vertices of degree five in $K_{2,5}$.
Figure \ref{fig:example_graph} also
illustrates our edge labelling.

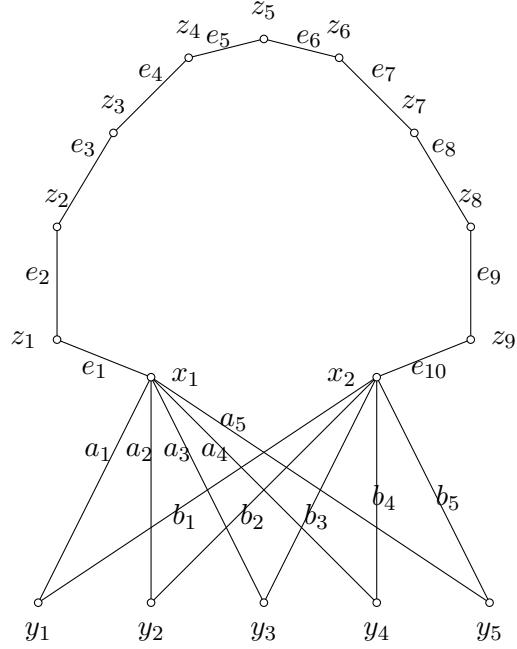
\begin{figure}
\centering
\begin{tikzpicture}[scale=0.50]
% [scale=.55,auto=left,every node/.style={circle,fill=lightgray, scale=.6, minimum size=3.2em}]
 %\fill[fill=white,draw=black] (1,2) circle (.1)
%node[label=above:$x_1$] {};

%\draw (3,6) -- (3,9) -- (6,12) -- (9,9) -- (9,6);

\draw (3,6) -- (0,0);
\draw (3,6) -- (3,0);
\draw (3,6) -- (6,0);
\draw (3,6) -- (9,0);
\draw (3,6) -- (12,0);

\draw (9,6) -- (0,0);
\draw (9,6) -- (3,0);
\draw (9,6) -- (6,0);
\draw (9,6) -- (9,0);
\draw (9,6) -- (12,0);

%\draw (3,6) -- (1.5,9) -- (4,13) -- (6,14.5) -- (8,13) -- (10.5,9) -- (9,6);
\draw (3,6) -- (.5,7) -- (.5,10) -- (2,12.5) -- (4,14.5) -- (6,15) -- (8,14.5) -- (10,12.5) -- (11.5,10) -- (11.5,7) -- (9,6);

\fill[fill=white,draw=black] (3,6) circle (.1) node[label=right:$x_1$] {};
\fill[fill=white,draw=black] (9,6) circle (.1) node[label=left:$x_2$] {};
\fill[fill=white,draw=black] (0,0) circle (.1) node[label=below:$y_1$] {};
\fill[fill=white,draw=black] (3,0) circle (.1) node[label=below:$y_2$] {};
\fill[fill=white,draw=black] (6,0) circle (.1) node[label=below:$y_3$] {};
\fill[fill=white,draw=black] (9,0) circle (.1) node[label=below:$y_4$] {};
\fill[fill=white,draw=black] (12,0) circle (.1) node[label=below:$y_5$] {};
%\fill[fill=white,draw=black] (3,9) circle (.1) node[label=above:$z_{1,1}$] {};
%\fill[fill=white,draw=black] (6,12) circle (.1) node[label=above:$z_{1,2}$] {};
%\fill[fill=white,draw=black] (9,9) circle (.1) node[label=above:$z_{1,3}$] {};

%\fill[fill=white,draw=black] (1.5,9) circle (.1) node[label=above:$z_{2,1}$] {};
%\fill[fill=white,draw=black] (4,13) circle (.1) node[label=above:$z_{2,2}$] {};
%\fill[fill=white,draw=black] (6,14.5) circle (.1) node[label=above:$z_{2,3}$] {};
%\fill[fill=white,draw=black] (8,13) circle (.1) node[label=above:$z_{2,4}$] {};
%\fill[fill=white,draw=black] (10.5,9) circle (.1) node[label=above:$z_{2,5}$] {};

\fill[fill=white,draw=black] (.5,7) circle (.1) node[label=left:$z_{1}$] {};
\fill[fill=white,draw=black] (.5,10) circle (.1) node[label=above:$z_{2}$] {};
\fill[fill=white,draw=black] (2,12.5) circle (.1) node[label=above:$z_{3}$] {};
\fill[fill=white,draw=black] (4,14.5) circle (.1) node[label=above:$z_{4}$] {};
\fill[fill=white,draw=black] (6,15) circle (.1) node[label=above:$z_{5}$] {};
\fill[fill=white,draw=black] (8,14.5) circle (.1) node[label=above:$z_{6}$] {};
\fill[fill=white,draw=black] (10,12.5) circle (.1) node[label=above:$z_{7}$] {};
\fill[fill=white,draw=black] (11.5,10) circle (.1) node[label=above:$z_{8}$] {};
\fill[fill=white,draw=black] (11.5,7) circle (.1) node[label=right:$z_{9}$] {};

\node (a1) at (1.6,4) {\text{$a_{1}$}};
\node (a2) at (2.7,4) {\text{$a_{2}$}};
\node (a3) at (3.7,4) {\text{$a_{3}$}};
\node (a4) at (4.7,4) {\text{$a_{4}$}};
\node (a5) at (5.2,4.8) {\text{$a_{5}$}};
\node (b1) at (3.9,2.3) {\text{$b_{1}$}};
\node (b2) at (5.7,2.3) {\text{$b_{2}$}};
\node (b3) at (7.4,2.3) {\text{$b_{3}$}};
\node (b4) at (9.2,2.8) {\text{$b_{4}$}};
\node (b5) at (10.9,2.8) {\text{$b_{5}$}};
%\node (e11) at (2.7,7.4) {\text{$e_{1,1}$}};
%\node (e12) at (4.2,10.7) {\text{$e_{1,2}$}};
%\node (e14) at (9.4,7.4) {\text{$e_{1,4}$}};
%\node (e13) at (7.7,10.7) {\text{$e_{1,3}$}};

%\node (e21) at (1.7,7.4) {\text{$e_{2,1}$}};
%\node (e22) at (2.1,10.7) {\text{$e_{2,2}$}};
%\node (e23) at (4.7,14.1) {\text{$e_{2,3}$}};
%\node (e24) at (7,14.1) {\text{$e_{2,4}$}};
%\node (e26) at (10.4,7.4) {\text{$e_{2,6}$}};
%\node (e25) at (9.7,10.7) {\text{$e_{2,5}$}};

\node (e31) at (1.5,6.2) {\text{$e_{1}$}};
\node (e32) at (0,8.7) {\text{$e_{2}$}};
\node (e33) at (1.2,12.1) {\text{$e_{3}$}};
\node (e34) at (3,14.1) {\text{$e_{4}$}};
\node (e35) at (4.8,15) {\text{$e_{5}$}};
\node (e36) at (7.2,15) {\text{$e_{6}$}};
\node (e37) at (9.2,14.1) {\text{$e_{7}$}};
\node (e38) at (10.8,12.1) {\text{$e_{8}$}};
\node (e39) at (12,8.7) {\text{$e_{9}$}};
\node (e310) at (10.4,6.2) {\text{$e_{10}$}};

\end{tikzpicture}
\caption{Illustration of $G_{r,d}=G_{6,5}$. This graph
is a $K_{2,5}$ with a path of length 
$2\cdot6 -2=10$ connecting the two
vertices of degree 5 in $K_{2,5}$.}\label{fig:example_graph} 
\end{figure}

\begin{rem}
    Note  Definition \ref{defn:specialgraph} still
    makes sense if we allow $r=2$.  However, if
    we allow $r =2$, then we are adding a path
    of length $2r - 2 = 2$ between $x_1$ and $x_2$.
    Then the graph $G_{2,d}$ and
    the graph $K_{2,d+1}$ are isomorphic.
    It thus makes sense to restrict to
    the case $r \geq 3$.
\end{rem}

The toric ideals of these graphs were studied
in \cite{GHKKPVT} (see also \cite{NN} for a 
more general family); in particular,
all the graded Betti numbers of $I_G$ were determined, and consequently, one can
determine the regularity 
(see \cite[Theorem 3.9]{GHKKPVT}).
The approach
of \cite{GHKKPVT} is to consider 
a careful analysis of
the initial ideal of $I_G$, and 
exploiting the fact that a universal
Gr\"obner basis of $I_G$ could be 
explicitly described:

\begin{lem}[{\cite[Corollary 3.3]{GHKKPVT}}]
\label{lem.ugb}
Let $d,r$ be integers such that 
$d \geq 1$ and $r \geq 3$.  Then 
a universal Gr\"obner basis for $I_{G_{r,d}}$
is given by
$$\{a_ib_j -a_jb_i ~|~ 1 \leq i < j \leq d\} 
\cup \{a_ie_2e_4\cdots e_{2r-2}-b_i
e_1e_3\cdots e_{2r-3} ~|~ 1 \leq i \leq d\},$$
where if $d=1$, the first set is empty.
\end{lem}

On the other hand,
since $G = G_{r,d}$ is bipartite, 
we know that $I_G$ is geometrically vertex decomposable.  Using this
fact, Lemma \ref{lem.ugb}, and Theorem \ref{thm.regformula}, 
we can give an alternative proof for the 
calculation of ${\rm reg}(I_{G_{r,d}})$.

%In order to achieve this goal, we need to label
%our edges.  For $i =1,\ldots,d$, let 
%$a_i = \{x_1,y_i\}$ and $b_i = \{x_2,y_i\}$.
%For $i=1,\ldots,t$, let $e_{1} = \{x_1,z_{1}\}$
%and $e_{2r-2} = \{x_2,z_{2r-3}\}$.  Finally,
%let $$e_{2} = \{z_{1},z_{2}\}, 
%e_{3} = \{z_{2},z_{3}\}, \ldots, e_{2r-3}
%= \{z_{2r-4},z_{2r-3}\}$.
%See Figure \ref{fig:example_graph} for an example
%of this labelling.

%\begin{lem}
%With the notation as above, let 
%$G = G_{c_1,\ldots,c_t,d}$.  The induced
%cycles of $G$, up to permutation, are
%\begin{enumerate}
%    \item $(a_i,b_i,b_j,a_j)$ with $1 \leq i < j \leq d$;
%    \item %$(a_i,b_i,e_{j,1},e_{j,2},\ldots,e_{j,c_j})$
%    for $1 \leq i \leq d$ and $1 \leq j \leq t$; and
%    \item %$(e_{i,1},e_{i,2},\ldots,e_{i,c_i},e_{j,c_j},e_{j,c_{j-1}},
%    \ldots,e_{j_1})$ with $1 \leq i< j \leq t$.
%\end{enumerate}
%\end{lem}

%\begin{cor}
%    With the notation as above, let 
%    $G = G_{c_1,\ldots,c_t,d}$.  Then
%    a universal Gr\"obner basis for $I_G$ is given %by
%    $$\{a_ib_j - a_jb_i ~|~ 1 \leq i < j \leq n\} %\cup$$
%    $$\{a_ie_{j,1}e_{j,3}\cdots e_{j,c_j-1} -
%    b_ie_{j,2}e_{j,4}\cdots e_{j,c_j} ~|~ 
%    1 \leq i \leq d, ~ 1 \leq j \leq t\} \cup$$
%    $$\{e_{i,1}e_{i,3}\cdots e_{i,c_i-1}e_{j,c_j}
%    e_{j,c_j-2}\cdots e_{j,2} - e_{i_2}e_{i,4}
%    \cdots e_{i,c_i}e_{j,c_j-1}e_{j,c_j-3}\cdots
%    e_{j,1} ~|~ 1 \leq i < j \leq t\}.$$
%\end{cor}

\begin{thm}\label{thm.specialclass}
Let $d \geq 1$ and $r \geq 3$ be integers
and let $G = G_{r,d}$.  Then
${\rm reg}(I_{G}) = r$.
\end{thm}

\begin{proof}
    We first consider the case that $d=1$
    and $r \geq 3$.  In this case
    $G_{r,1} = C_{2r}$, that is, the 
    cycle on $2r$ vertices. By Lemma 
    \ref{lem.ugb}, $I_{G_{r,1}} = \langle a_1e_2\cdots e_{2r-2} - b_1e_1\cdots e_{2r-3} \rangle$ is generated by a single polynomial of degree $r$.  The 
    conclusion then follows.

    We now proceed by induction on the
    tuple $(d,r)$, where we assume the
    result holds for all graphs $G_{r',d'}$ 
    with $d >d'$.  Let $G = G_{r,d}$.
    Let $>$ denote the  lexicographical monomial order
    on $\mathbb{K}[E(G)]$ with
    $$a_d > a_{d-1} > \cdots > a_1 > 
    e_1 > e_2 > \cdots > e_{2r-2} > b_d > \cdots
    >b_1.$$
    If $y=a_d$, then $>$ is a $y$-compatible
    monomial order. 
    
    By Lemma \ref{lem:nyideal},
    $N_{y,I_G}$ is the toric ideal
    of $G \setminus \{a_d\}$.  If we remove
    $a_d$ from $G$, then $b_d$ is a leaf
    of $G \setminus \{a_d\}$. Consequently,
    $N_{y,I_G}$ is the toric ideal of 
    the graph $G$ with both edges $a_d$ and
    $b_d$ removed.  But if remove
    $a_d$ and $b_d$ from $G$, we obtain the
    graph $G_{r,d-1}$.  Thus $N_{y,I_G}
    = I_{G_{r,d-1}}$.

    By using the universal
    Gr\"obner basis of Lemma \ref{lem.ugb} we
    have 
    \begin{eqnarray*}
    C_{y,I_G} &=& \langle b_{d-1},\ldots,b_1,
    e_{2}e_4\cdots e_{2r-2} \rangle  
    + \langle a_ib_j - a_jb_i ~|~ 1 \leq i < j \leq 
    d-1 \rangle + \\
    &&\langle a_ie_2\cdots e_{2r-2}
    - b_ie_1\cdots e_{2r-3} ~|~ 1 \leq i \leq d-1\} \\
    &=& \langle b_{d-1},\ldots,b_1,e_2e_4\cdots e_{2r-2} \rangle.
    \end{eqnarray*} 
    The last equality follows from the 
    fact that each term of the generators
    in the other two
    ideals is either divisible by some $b_i$
    with $i \in \{1,\ldots,d-1\}$ or $e_2e_4\cdots e_{2r-2}$.  
    Consequently, $C_{y,I_G}$ is a monomial
    ideal that is a complete intersection (since
    the monomials have disjoint support)
    with regularity
    $\underbrace{1+\cdots+1}_{d-1} + (r-1) - (d-1) = r-1$.  
    
    So, by Theorem \ref{thm.regformula}
    and induction we have
    \begin{eqnarray*}
        {\rm reg}(I_G) = \max\{{\rm reg}
        (I_{G_{r,d-1}}), {\rm reg}(C_{y,I_G})+1\} =\max\{r,(r-1)+1\} = r,
    \end{eqnarray*}
    as desired.
    \end{proof}

We can now compute
the $a$-invariant and the multiplicity
of the rings $\kk[E]/I_{G_{r,d}}$.

\begin{cor}
    Let $d \geq 1$ and $r \geq 3$ be integers
and let $G = G_{r,d}$.  Then
\begin{enumerate}
\item $a(\mathbb{K}[E(G)]/I_G) = 1-d-r$ and
\item $e(\mathbb{K}[E(G)]/I_G) = dr-(d-1).$
\end{enumerate}
\end{cor}

\begin{proof}
We can use Theorem \ref{thm.aformula} 
to prove $(1)$, but it is more direct
to use 
Theorem \ref{thm.specialclass}, Lemma \ref{lem.ainvarianttoric},
and the fact that $G_{r,d}$ is a
bipartite graph on $2+d+(2r-3)$ vertices.

To prove $(2)$ we do induction
on $d$.  If $d=1$, then $G_{r,d} = C_{2r}$,
and so $I_{G_{r,d}}$ is a principal
ideal generated by a single generator
of degree $r$.  So $e(\kk[E]/I_{G_{r,d}}) =
r$.

So suppose that $d>1$. 
If $I = I_{G_{r,d}}$, then by Theorem
\ref{thm.eformula} we have
$$e(\kk[E]/I) = e(\kk[E]/N_{y,I}) +
e(\kk(E]/C_{y,I}).$$
As shown in the proof of Theorem \ref{thm.specialclass}, $C_{y,I}$ is
a complete intersection generated by
$d-1$ generators of degree one and 
one generator of degree $r-1$.  
Consequently, $e(\kk[E]/C_{y,I}) = 1^{d-1}\cdot (r-1) = r-1$. On the other hand, as also shown
in the proof Theorem \ref{thm.specialclass},
$N_{y,I} = I_{G_{r,d-1}}$.  Hence, by
induction $e(\kk[E]/N_{y,I}) = (d-1)r-(d-2)$.  Thus $e(\kk[E]/I) = (d-1)r-(d-2)+r-1 = dr-(d-1).$
\end{proof}

\noindent
{\bf Acknowledgments.}
The computer program {\tt Macaulay2} \cite{M2}
was used for computations and examples; in 
particular, we made use of the
package {\tt GeometricDecomposability} \cite{CVT}.
The authors thank Patricia Klein for 
her comments and  suggestions.
Rajchgot's research is supported
by NSERC Discovery Grant 2017-05732.
 Van Tuyl’s research is supported by NSERC Discovery Grant 2019-05412. 
 
\bibliographystyle{plain}
\bibliography{reference}

\begin{thebibliography}{10}

\bibitem{AK1989}
Shreeram Abhyankar and Devadatta~M. Kulkarni.
\newblock On {H}ilbertian ideals.
\newblock {\em Linear Algebra Appl.}, 116:53--79, 1989.

\bibitem{ADS22}
Ayah Almousa, Anton Dochtermann, and Ben Smith.
\newblock Root polytopes, tropical types, and toric edge ideals.
\newblock 2022.
\newblock Preprint, {\tt arXiv:2209.09851}.

\bibitem{BOVT17}
Jennifer Biermann, Augustine O'Keefe, and Adam Van~Tuyl.
\newblock Bounds on the regularity of toric ideals of graphs.
\newblock {\em Adv. in Appl. Math.}, 85:84--102, 2017.

\bibitem{BH}
Winfried Bruns and J\"{u}rgen Herzog.
\newblock {\em Cohen--{M}acaulay rings}, volume~39 of {\em Cambridge Studies in
  Advanced Mathematics}.
\newblock Cambridge University Press, Cambridge, 1993.

\bibitem{CG}
Alexandru Constantinescu and Elisa Gorla.
\newblock Gorenstein liaison for toric ideals of graphs.
\newblock {\em J. Algebra}, 502:249--261, 2018.

\bibitem{CN09}
Alberto Corso and Uwe Nagel.
\newblock Monomial and toric ideals associated to {F}errers graphs.
\newblock {\em Trans. Amer. Math. Soc.}, 361(3):1371--1395, 2009.

\bibitem{CDHR23}
Mike Cummings, Sergio Da~Silva, Megumi Harada, and Jenna Rajchgot.
\newblock Gr\"obner geometry for regular nilpotent {H}essenberg {S}chubert
  cells.
\newblock 2023.
\newblock Preprint, {\tt arXiv:2305.19335}.

\bibitem{CSRV22}
Mike Cummings, Sergio Da~Silva, Jenna Rajchgot, and Adam Van~Tuyl.
\newblock Geometric vertex decomposition and liaison for toric ideals of
  graphs.
\newblock {\em Algebr. Comb.}, 6(4):965--997, 2023.

\bibitem{CVT}
Mike Cummings and Adam Van~Tuyl.
\newblock The {G}eometric{D}ecomposability package for {M}acaulay2.
\newblock 2023.
\newblock Preprint, {\tt arXiv:2211.02471}.

\bibitem{DH}
Sergio Da~Silva and Megumi Harada.
\newblock Geometric vertex decomposition, gr{\"o}bner bases, and {F}robenius
  splittings for regular nilpotent {H}essenberg varieties.
\newblock {\em Transformation Groups}, pages 1--36, 2023.
\newblock \url{doi.org/10.1007/s00031-023-09808-1 }.

\bibitem{DA2015}
Alessio D'Al\`{i}.
\newblock Toric ideals associated with gap-free graphs.
\newblock {\em J. Pure Appl. Algebra}, 219(9):3862--3872, 2015.

\bibitem{DE09}
Anton Dochtermann and Alexander Engstr\"{o}m.
\newblock Algebraic properties of edge ideals via combinatorial topology.
\newblock {\em Electron. J. Combin.}, 16(2):Research Paper 2, 24, 2009.

\bibitem{EH}
Viviana Ene and J\"{u}rgen Herzog.
\newblock {\em Gr\"{o}bner bases in commutative algebra}, volume 130 of {\em
  Graduate Studies in Mathematics}.
\newblock American Mathematical Society, Providence, RI, 2012.

\bibitem{FHKV21}
Giuseppe Favacchio, Johannes Hofscheier, Graham Keiper, and Adam Van~Tuyl.
\newblock Splittings of toric ideals.
\newblock {\em J. Algebra}, 574:409--433, 2021.

\bibitem{GHKKPVT}
Federico Galetto, Johannes Hofscheier, Graham Keiper, Craig Kohne, Adam
  Van~Tuyl, and Miguel Eduardo~Uribe Paczka.
\newblock Betti numbers of toric ideals of graphs: a case study.
\newblock {\em J. Algebra Appl.}, 18(12):1950226, 14, 2019.

\bibitem{M2}
Daniel~R. Grayson and Michael~E. Stillman.
\newblock Macaulay2, a software system for research in algebraic geometry.
\newblock Available at \url{http://www2.macaulay2.com}.

\bibitem{HBO2019}
Huy~T\`ai H\`a, Selvi~Kara Beyarslan, and Augustine O'Keefe.
\newblock Algebraic properties of toric rings of graphs.
\newblock {\em Comm. Algebra}, 47(1):1--16, 2019.

\bibitem{HW2014}
Huy~T\`ai H\`a and Russ Woodroofe.
\newblock Results on the regularity of square-free monomial ideals.
\newblock {\em Adv. in Appl. Math.}, 58:21--36, 2014.

\bibitem{HHO}
J\"{u}rgen Herzog, Takayuki Hibi, and Hidefumi Ohsugi.
\newblock {\em Binomial ideals}, volume 279 of {\em Graduate Texts in
  Mathematics}.
\newblock Springer, Cham, 2018.

\bibitem{MR4635090}
Patricia Klein.
\newblock Diagonal degenerations of matrix {S}chubert varieties.
\newblock {\em Algebr. Comb.}, 6(4):1073--1094, 2023.

\bibitem{KR21}
Patricia Klein and Jenna Rajchgot.
\newblock Geometric vertex decomposition and liaison.
\newblock {\em Forum Math. Sigma}, 9:Paper No. e70, 23, 2021.

\bibitem{KW21}
Patricia Klein and Anna Weigandt.
\newblock Bumpless pipe dreams encode {G}r\"{o}bner geometry of {S}chubert
  polynomials, 2021.

\bibitem{KMY09}
Allen Knutson, Ezra Miller, and Alexander Yong.
\newblock Gr\"{o}bner geometry of vertex decompositions and of flagged
  tableaux.
\newblock {\em J. Reine Angew. Math.}, 630:1--31, 2009.

\bibitem{MSbook}
E.~Miller and B.~Sturmfels.
\newblock {\em Combinatorial Commutative Algebra}, volume 227 of {\em Graduate
  Texts in Mathematics}.
\newblock Springer-Verlag, New York, 2005.

\bibitem{MK16}
Somayeh Moradi and Fahimeh Khosh-Ahang.
\newblock On vertex decomposable simplicial complexes and their {A}lexander
  duals.
\newblock {\em Math. Scand.}, 118(1):43--56, 2016.

\bibitem{NN}
Rimpa Nandi and Ramakrishna Nanduri.
\newblock Betti numbers of toric algebras of certain bipartite graphs.
\newblock {\em J. Algebra Appl.}, 18(12):1950231, 18, 2019.

\bibitem{OH}
Hidefumi Ohsugi and Takayuki Hibi.
\newblock Koszul bipartite graphs.
\newblock {\em Adv. in Appl. Math.}, 22(1):25--28, 1999.

\bibitem{PB80}
J.~Scott Provan and Louis~J. Billera.
\newblock Decompositions of simplicial complexes related to diameters of convex
  polyhedra.
\newblock {\em Math. Oper. Res.}, 5(4):576--594, 1980.

\bibitem{S78}
Richard~P. Stanley.
\newblock Hilbert functions of graded algebras.
\newblock {\em Advances in Math.}, 28(1):57--83, 1978.

\bibitem{S}
Richard~P. Stanley.
\newblock A monotonicity property of {$h$}-vectors and {$h^*$}-vectors.
\newblock {\em European J. Combin.}, 14(3):251--258, 1993.

\bibitem{SY23}
Ada Stelzer and Alexander Yong.
\newblock Schubert determinantal ideals are {H}ilbertian.
\newblock 2023.
\newblock Preprint, {\tt arXiv:2305.12558}.

\bibitem{TT2011}
Christos Tatakis and Apostolos Thoma.
\newblock On the universal {G}r\"{o}bner bases of toric ideals of graphs.
\newblock {\em J. Combin. Theory Ser. A}, 118(5):1540--1548, 2011.

\bibitem{Vas}
Wolmer~V. Vasconcelos.
\newblock {\em Computational methods in commutative algebra and algebraic
  geometry}, volume~2 of {\em Algorithms and Computation in Mathematics}.
\newblock Springer-Verlag, Berlin, 1998.
\newblock With chapters by David Eisenbud, Daniel R. Grayson, J\"{u}rgen Herzog
  and Michael Stillman.

\bibitem{VV23}
Maria Vaz~Pinto and Rafael~H. Villarreal.
\newblock Graph rings and ideals: {W}olmer {V}asconcelos contributions.
\newblock 2023.
\newblock Preprint, {\tt arXiv:2305.06270}.

\bibitem{Vart}
Rafael~H. Villarreal.
\newblock Rees algebras of edge ideals.
\newblock {\em Comm. Algebra}, 23(9):3513--3524, 1995.

\bibitem{V}
Rafael~H. Villarreal.
\newblock {\em Monomial algebras}.
\newblock Monographs and Research Notes in Mathematics. CRC Press, Boca Raton,
  FL, second edition, 2015.

\bibitem{W09}
Russ Woodroofe.
\newblock Vertex decomposable graphs and obstructions to shellability.
\newblock {\em Proc. Amer. Math. Soc.}, 137(10):3235--3246, 2009.

\end{thebibliography}
\end{document}